\documentclass[a4paper,11pt,leqno,english]{smfart}
\usepackage{aeguill}
\usepackage{enumerate}
\usepackage{amssymb,amsmath,latexsym,amsthm}
\usepackage[T1]{fontenc}
\usepackage{geometry}
\usepackage{url}
\usepackage[french, main=english]{babel}
\usepackage[utf8]{inputenc}
\usepackage{mathrsfs}
\usepackage{xcolor}
\usepackage{comment}
\definecolor{violet}{rgb}{0.0,0.2,0.7}
\definecolor{rouge2}{rgb}{0.8,0.0,0.2}
\usepackage{tikz}
\usepackage{empheq}
\usepackage{tikz-cd}
\usetikzlibrary{matrix,arrows,decorations.pathmorphing}
\usepackage{hyperref}
\usepackage{mathpazo}
\usepackage[colorinlistoftodos,prependcaption]{todonotes}

\hypersetup{
    bookmarks=true,         
    unicode=false,          
    pdftoolbar=true,        
    pdfmenubar=true,        
    pdffitwindow=false,     
    pdfstartview={FitH},    
    pdftitle={},    
    pdfauthor={},     
    colorlinks=true,       
   linkcolor=rouge2,          
    citecolor=violet,        
    filecolor=black,      
    urlcolor=cyan}           
\usepackage{enumitem}
\usepackage{appendix}

 \theoremstyle{plain}    
 \newtheorem{thm}{Theorem}[section]
\theoremstyle{plain} 
\newtheorem{bigthm}{Theorem}
\newtheorem{bigcoro}[bigthm]{Corollary}

\newtheorem{ques}{Question}

 \numberwithin{equation}{section} 
 \numberwithin{figure}{section} 
 \newtheorem{cor}[thm]{Corollary} 
 \theoremstyle{plain}    
 \newtheorem{prop}[thm]{Proposition} 
 \theoremstyle{plain}    
 \newtheorem{lem}[thm]{Lemma} 
 \theoremstyle{remark}
  \newtheorem{claim}[thm]{Claim} 
 \theoremstyle{remark}
 \newtheorem{rem}[thm]{Remark}
 \theoremstyle{definition}
\newtheorem{exa}[thm]{Example}
\theoremstyle{plain}  

\theoremstyle{plain}

\theoremstyle{definition}
\newtheorem{defi}[thm]{Definition}

\newtheorem*{ackn}{Acknowledgements}

\newcommand{\C}{{\mathbb{C}}}
\newcommand{\N}{{\mathbb{N}}}

\newcommand{\R}{{\mathbb{R}}}

\newcommand{\cS}{{\mathcal{S}}}

\def\1{\mathbf{1}}

\newcommand{\wt}{\widetilde}

\newcommand{\e}{\varepsilon}
\newcommand{\om}{\omega}
\newcommand{\f}{\varphi}

\newcommand{\p}{\psi}

\newcommand{\ep}{\varepsilon}

\newcommand{\Ric}{\mathrm{Ric} \,}

\newcommand{\cR}{\mathcal R}

\renewcommand{\ge}{\geq}
\renewcommand{\le}{\leq}

\newcommand{\tr}{\operatorname{tr}}

\newcommand{\PSH}{\operatorname{PSH}}

\newcommand{\Rm}{\operatorname{Rm}}

\title{K\"ahler-Ricci flows coming out of Metric spaces}
\date{\today}

\author{Alix Deruelle}

\address{Laboratoire de Mathématiques d'Orsay \& Institut Universitaire de France; 
UMR 8628, Université Paris-Saclay, CNRS, Bâtiment 307, rue Michel Magat, 91405, Orsay, France}

\email{alix.deruelle@universite-paris-saclay.fr}

\author{Vincent Guedj}

\address{Institut de Mathématiques de Toulouse \& Institut Universitaire de France; UMR 5219, Université de Toulouse; CNRS, UPS, 118 route de Narbonne, F-31062 Toulouse Cedex 9, France}

\email{vincent.guedj@math.univ-toulouse.fr}

\author{Henri Guenancia}

\address{Univ. Bordeaux, CNRS, Bordeaux INP, IMB, UMR 5251, F-33400 Talence, France}
\email{henri.guenancia@math.cnrs.fr}

\author{Ahmed Zeriahi}

\address{Institut de Mathématiques de Toulouse; UMR 5219, Université de Toulouse; CNRS, UPS, 118 route de Narbonne, F-31062 Toulouse Cedex 9, France}

\email{ahmed.zeriahi@math.univ-toulouse.fr}

\begin{document}

\begin{abstract}  
Given a compact Kähler manifold $X$ and a closed, positive $(1,1)$-current $T$ on $X$, we find sufficient conditions for $T$ to induce a metric structure $(X,d_T)$ which is the Gromov-Hausdorff limit of compact Kähler manifolds either in a "static" way or at time zero of smooth K\"ahler-Ricci flows. In dimension $1$ we extend works of T. Richard and M. Simon, showing that any oriented compact Alexandrov surface with bounded integral curvature and without cusp is the initial datum of a K\"ahler-Ricci flow.
\end{abstract} 

\maketitle

\tableofcontents

\section{Introduction}

   \subsection{Overview}
  A (noncollapsed) Ricci limit space of dimension $n$
 is a compact metric space which is the Gromov-Hausdorff
limit of non-collapsing 
Riemannian manifolds $(M_i,g_i)$ of dimension $n$
 with uniformly bounded Ricci curvature,
$$
{\rm Ric}(g_i) \geq -K
\; \; \; \text{ and } \; \; \;
{\rm Vol} \, B_{g_i}(x_i,1) \geq v_0>0.
$$

  A long standing conjecture  attributed to Anderson, Cheeger, Colding and Tian  (see \cite{ChCo97})
  asserts that    a Ricci limit space is homeomorphic to a smooth manifold outside a singular set
   of Hausdorff dimension $n-4$. This conjecture has been settled 
   recently by Simon and Topping  \cite{Sim12,ST17} in dimension $3$
   (we shall discuss in some detail the $2$-dimensional case in section \ref{sec:alexandrov},
   following the work of Richard \cite{Rich18}), by 
running smooth Ricci flows from these Ricci limit spaces. See also \cite{Naber-Open-Pb} for an update on this conjecture.\\

    It is a natural problem  to try and characterize those  metric spaces that
   can arise as initial data of smooth Ricci flows, where
 the convergence at time zero should hold in the Gromov-Hausdorff sense.
We study in this article the K\"ahler version of this problem: we consider 
{\it K\"ahler Ricci limit spaces} (the approximating manifolds are assumed
to be {\it compact K\"ahler}),
and take advantage of the strong analytic rigidity properties of
such families (see \cite{DS14,DS17,LS18}) to relax the Ricci lower bound.

   Motivated by the previous discussion, let $(X,\omega_X)$ be a compact K\"ahler manifold of complex dimension $n$,
and let $T$ be a positive closed current cohomologous to $\omega_X$.
 It follows from the $\partial\overline{\partial}$-lemma that $T=\omega_X+dd^c \f$
for some {\it quasi-plurisubharmonic} (qpsh) function $\f$. Recall that a function
is qpsh if it is locally the sum of a smooth and a plurisubharmonic function.
Here $d^c=\frac{1}{2i \pi}(\partial-\overline{\partial})$ 
 so that $dd^c =\frac{i}{\pi}\partial\overline{\partial}$.
We assume in this article that $\f$
is {\it continuous}, so that the Monge-Amp\`ere measure $T^n$ is well-defined in the weak 
sense of currents (see \cite{GZbook}). We moreover assume that 
$T^n=f \omega_X^n$ is absolutely continuous with respect to Lebesgue measure,
with $f=e^{\p_+-\p_-}$ for some qpsh functions $\p_{\pm}$,
with $e^{-\p_-} \in L^1(\omega_X^n)$.

\smallskip

 We first address the following general problem:

 \begin{ques}\label{ques-1}
 Under which sufficient conditions does the current $T$ define a natural (semi)distance $d_T$  on $X$, 
 and what are the geometric properties of $d_T$ ?
 \end{ques}
 
 Observe that we have not defined what the semi-distance $d_T$ induced by $T$ is yet and we believe that such a definition is part of the difficulty to answer Question \ref{ques-1}. Indeed, defining $d_T(x,y)$ for two points $x$ and $y$ in $X$ as the infimum
  of the lengths induced by $T$ of Lipschitz curves joining $x$ to $y$ seems too naive to even make sense in general.
   Nonetheless, we are able to answer Question \ref{ques-1} both in complex dimension $n=1$ in full generality and in the setting of \textit{geometric} isolated singularities (see Definition \ref{def:geomsing}) in  complex dimensions $n\geq 2$. 
The proofs are independent of the K\"ahler-Ricci flow.

 \smallskip
 
It has been shown in \cite{GZ17,DnL17} that there exists a unique K\"ahler-Ricci flow 
on $X$ with $T$ as initial condition, i.e. a family of 
K\"ahler forms $(\omega_t)_{t \in (0,T_{max})}$, such that
$$
\frac{\partial \omega_t}{\partial t}=-{\rm Ric}(\omega_t)
\; \; \text{ and } \; \;
\omega_t \stackrel{t \rightarrow 0}{\longrightarrow} T.
$$
The maximal existence time only depends on the cohomology class of the initial datum,
$$
T_{\rm max}:=\sup \{ t>0 , \; \{\omega_X\} -tc_1(X)  \text{ is K\"ahler } \}.\\
$$
The convergence at time zero holds in the weak sense of currents. A second aspect of this article is to reinforce the latter, 
when the initial datum induces a metric space.
We  try and establish Gromov-Hausdorff convergence at time zero whenever possible.

To slightly simplify the analysis, we actually consider the following twisted version 
\begin{equation}\label{eqn-krf-twist-shout}
\frac{\partial \omega_t}{\partial t}=-{\rm Ric}(\omega_t)+{\rm Ric}(\omega_X)
\; \; \text{ and } \; \;
\omega_t \stackrel{t \rightarrow 0}{\longrightarrow} T.
\end{equation}
In this case $T_{\rm max}=+\infty$  and the results 
of \cite{GZ17,DnL17} apply.
Similar to the discussion on the current $T$ above, it holds that $\omega_t=\omega_X+dd^c \f_t$ and
the twisted K\"ahler-Ricci flow equation
is equivalent to the  scalar parabolic complex Monge-Amp\`ere equation
$$
(\omega_X+dd^c \f_t)^n=e^{\dot{\f_t}} \omega_X^n
\; \; \text{ and } \; \;
\f_t \stackrel{t \rightarrow 0}{\longrightarrow} \f.
$$
 We then address the following second general problem:

 \begin{ques}\label{ques-2}
 Under which sufficient conditions on $T$ does the K\"ahler-Ricci flow $(X,\omega_t)$ Gromov-Hausdorff converge to $(X,d_T)$ as $t\rightarrow 0$ ?
 \end{ques}
 
We are able to answer Question \ref{ques-2} both in complex dimension $n=1$ in full generality and in the setting of \textit{geometric} isolated singularities 
in
complex dimensions $n\geq 2$.

\subsection{Main results}

Our first result analyses the definition and properties of the semi-distance induced by 
$T$ on $X$,
when $T$ has {\it geometric singularities}:

\begin{defi} 
\label{def:geomsing}
We say that a positive closed $(1,1)$-current $T=\omega_X+dd^c \f$ has geometric singularities if
there are qpsh functions $\p_{\pm}$ on $X$
such that
\begin{itemize}
\item[$\bullet$] the qpsh function $\f$ is continuous and $T^n=e^{\p_+-\p_-} \omega_X^n$ with $e^{-\p_-} \in L^1(\omega_X^n)$;  
\item[$\bullet$] the functions $\p_{\pm}$ are smooth in some dense Zariski open set $\Omega\subset X$.
\end{itemize}
\end{defi}

Currents with geometric singularities are ubiquituous in complex geometry. Many canonical currents,
e.g. singular K\"ahler-Einstein metrics \cite{EGZ09} or 
limits  of the K\"ahler-Ricci flow approach to the Minimal Model Program \cite{ST09}
are of this type. 
 We then set $Z =X \setminus \Omega$ and define the semi-distance induced by $T$ 
 as follows: 
 $$
d_{T}(x,y)=\inf \left\{ \ell_{T}(\gamma), \; 
\gamma\, \text{Lipschitz path from } x \text{ to } y \,\, \mbox{s.t.} \, \,\gamma^{-1}(Z) \, \mbox{is finite} \right\}.
$$

     \begin{bigthm} \label{thmE}
     Let $(X,\omega_X)$ be a compact K\"ahler manifold and let $T$ be a current
with geometric singularities. Then the following holds:
\begin{enumerate}[label=$(\roman*)$]
\item $d_T$ is a semi-distance and it is the unique continuous extension to $X\times X$ of the geodesic distance
induced by $T$ on $\Omega\times\Omega$;
\item There exists $C>0$ and $\alpha\in(0,1)$ such that $d_{T} \le C d_X^\alpha$. In particular $d_{T}$ is finite;  
\item If $\psi_+$ has isolated analytic singularities, then $d_T$ is bi-Hölder equivalent to $d_X$.
\end{enumerate}
\end{bigthm} 

In the proof of Theorem \ref{thmE}, it is further shown that if $\psi_+$ has isolated singularities that are not necessarily analytic then $d_{T}$ is a distance. In particular $(X,d_T)$ is isometric to the metric completion of the metric space associated to the Kähler manifold $(\Omega, T|_\Omega)$,  see Proposition~\ref{prop dT}.

As a combination of Theorem \ref{thmE} and the results of \cite{CCII}, \cite{LS18} and \cite{CJN}, we obtain the following corollary:
 \begin{bigcoro}\label{coro-B}
  Let $(X,\omega_X)$ be a compact K\"ahler manifold and let $T$ be a current
with geometric singularities. If  $\psi_+=0$ then $(X,d_T)$ is a Gromov-Hausdorff limit of K\"ahler manifolds with uniform Ricci lower bound. 
 \end{bigcoro}
 
 See Theorem~\ref{thm metric completion} for a more formal statement of Corollary \ref{coro-B} including the dimension bound $\dim_{\mathcal H} (Z,d_T)\le 2n-2$. Let us add that the results of \cite{G+} combined with Theorem \ref{thmE} (see Proposition~\ref{prop dT}) show that $(X,d_T,T^n)$ is a non-collapsed $\mathrm{RCD}(2n,\lambda)$ space for some $\lambda \in \R$. In turn, this enables to recover the dimension bound on $Z$ by appealing to \cite{Sz25}.
 
 \medskip
 
Our second series of results concern the metric behavior of solutions to the twisted K\"ahler-Ricci flow \eqref{eqn-krf-twist-shout} with initial condition $T$.
Let $d_t$ denote the Riemannian distance defined by the K\"ahler forms $\omega_t$.
The Ricci curvature of $\omega_t$ is the differential form
$$
{\rm Ric}(\omega_t)={\rm Ric}(\omega_X)-dd^c \log \left( \omega_t^n/\omega_X^n \right)={\rm Ric}(\omega_X)-dd^c \dot{\f_t},
$$
while   ${\rm Ric}(T):={\rm Ric}(\omega_X)-dd^c (\p_+-\p_-)$ is
a closed bidegree $(1,1)$-current of order zero
(difference of positive closed currents).

       \begin{bigthm} \label{thmB}
     Let $(X,\omega_X)$ be a compact K\"ahler manifold.  Let $T=\omega_X+dd^c\varphi$ be a positive current such that $\varphi$ is continuous, and $T^n=e^{\psi_+-\psi_{-}}\omega_X^n$ where $\p_{\pm}$ are qpsh functions 
     with $e^{-\psi_{-}}\in L^1(\omega_X^n)$.
   Let $(X,\omega_t)_{t>0}$ be the solution to the twisted K\"ahler-Ricci flow \eqref{eqn-krf-twist-shout}.  
   \begin{enumerate}[label=$(\roman*)$]
\item The Ricci curvatures ${\rm Ric}(\omega_t)$ weakly converge towards ${\rm Ric}(T)$ as $t \rightarrow 0$.

\smallskip

  \item   If $T$ has  geometric singularities and $\p_{\pm}$ have isolated singularities,
then the distances $d_{t}$ uniformly converge, as $t \rightarrow 0$, towards $d_{T}$.
In particular the compact metrics spaces $(X,d_{t})$ converge in the Gromov-Hausdorff sense
to  $(X,d_{T})$ as $t \rightarrow 0$.
\end{enumerate}
 \end{bigthm} 
 
  In the course of the proof of Theorem \ref{thmB}, 
   we provide crucial a priori $C^2$ estimates on the K\"ahler potential $\varphi_t$ together with an $L^{\infty}$ bound on $\dot{\varphi}_t$ which have their own interest.  
   
   Theorem \ref{thmB} gives   satisfactory existence results for smoothing K\"ahler-Ricci flows
coming out of K\"ahler-Ricci limit spaces. It would be desirable to better understand uniqueness
properties of these flows (see  \cite{TY24} for some uniqueness results when $n=1$).

\smallskip

Our last results concern Alexandrov surfaces. More precisely, when the complex dimension $n$ is equal to $1$, we provide a thorough answer to Questions \ref{ques-1} and \ref{ques-2}  through very precise estimates. If $\p_\pm$ are quasi-subharmonic functions on a compact Riemann surface $S$, we let
$$
\nu_{\pm}:=\sup \{ \nu(\p_{\pm},x), \; x \in S \}
$$
denote the maximal value of Lelong numbers of the functions $\p_{\pm}$.
Let us stress that the integrability property $e^{-\p_-} \in L^1$  is equivalent to the upper bound  $\nu_-<2$.

   \begin{bigthm} \label{thmD}
Let $(S,\omega_S)$ be a compact Riemann surface with induced distance $d_S$. Let $T=e^{\p_+-\p_-} \omega_S$ be a current such that $\p_{\pm}$ are qsh functions on $S$ with $e^{-\psi_{-}}\in L^1(\omega_S)$. 
 Then 
 \begin{itemize}
\item[$\bullet$]  $d_T$ is a well-defined distance which is bi-Hölder equivalent to $d_S$,
 \[C^{-1}\,d_{S}^{1+\alpha_+} \leq d_{T} \leq C\,d_{S}^{1-\alpha_-}, \]
 where $0<\e\ll1$, $\alpha_{\pm}:=\frac {1}{2}(\nu_{\pm}+\e)$ and $C=C(\om_S, \p_{\pm}, \e)>0$.
 
 \smallskip
 
 \item[$\bullet$]  $(S,d_T)$ is an oriented compact Alexandrov surface with bounded integral curvature
and without cusp, it is the initial datum of a K\"ahler-Ricci flow. 
\end{itemize}
\end{bigthm}

We refer the reader to Theorems \ref{dist dim 1}, \ref{thm:recapdim1}, and  \ref{thm:fkr1dim}  
for more complete results. Theorem \ref{thmD} extends the main results of Richard \cite{Rich18} and Simon \cite{Sim12},
who considered Alexandrov surfaces whose curvature is uniformly bounded below
(case $\p_+ \equiv 0$) to the case of bounded integral curvature with no cusp
(see Definition \ref{defn-cusp}).


 \subsection{Outline of the paper}
In Section \ref{sec-smoo-flow} we recall the basics of quasi-plurisubharmonic functions, as well as the background material on Monge-Amp\`ere measures induced by a positive closed current $T$ whose density is $L^p$ integrable for $p>1$, condition that we require throughout the text. Diameter bounds we need in the proof of our results are then stated in Section \ref{sec-diam-bd}. 
 Section \ref{cx-MA-flow}  covers the basic properties of the twisted K\"ahler-Ricci flow \eqref{eqn-krf-twist-shout} starting from a positive closed current $T$ which is cohomologous to a fixed background K\"ahler metric $\omega_X$, 
and has continuous K\"ahler potential.

Section \ref{sec:fkr}  establishes
 uniform a priori estimates along the 
 flow \eqref{eqn-krf-twist-shout} which we believe have their own interest (see Propositions \ref{pro:c0bounds}-\ref{pro:higher}, Lemma \ref{lem:bddphidot} and Theorem \ref{thm:laplace}):
\begin{itemize}
\item[$\bullet$] there exists $C_0>0$ such that $\|\f_t\|_{L^{\infty}(X)} \leq C_0$ for all $t>0$;
\item[$\bullet$] there exists $C_1>0$ such that
$\p_+(x)-C_1 \leq \dot{\f_t}(x) \leq -\p_-(x)+C_1$ for all $(t,x)$;
\item[$\bullet$]  there exists $C_2>0$ such that
$C_2^{-1} e^{\p_+} \omega_X \leq \omega_t \leq C_2e^{-\p_-} \omega_X$ on $X \times \R^+$.
\end{itemize}
When $\p_+=0$ and $\p_-$ has analytic singularities, we also obtain fine higher order estimates:
for each $k\geq 0$, there exists $B_k,\gamma_k>0$ such that 
\begin{itemize}
\item[$\bullet$] $|\nabla^{g(t)}\dot{\f_t}|_{g(t)}\leq B_0 e^{-\gamma_0 \psi_{-}+B_0t}$, 
\item[$\bullet$] $|\nabla^{g(t)}(\omega_t-\omega_X)|_{g(t)}\leq\frac{B_0}{\sqrt{t}}e^{-\gamma_0\psi_{-}+B_0t}$, and
\item[$\bullet$] $ |\nabla^{g(t),\,k} {\rm Rm}(g(t))|_{g(t)}\leq\frac{C_k}{t^{1+\frac{k}{2}}}e^{-\gamma_k\psi_{-}+C_kt}$,
\end{itemize}
as we explain in Propositions  \ref{pro:alix1} and \ref{pro:alix2}.


The content of Section \ref{sec:metricspace} takes on a more geometric flavour with the proof of Theorem~\ref{thmD} in Section \ref{sec:dim1}: see the crucial Proposition \ref{prop:ExpInteg} which establishes quantitative integrability properties of subharmonic functions with sufficiently small Lelong numbers in the plane along (real) curves. We then move to higher dimensional K\"ahler manifolds: the statement and proof of Theorem \ref{thmE} are made precise in Section \ref{higher dim} through Proposition \ref{prop dT}. The last Section \ref{sec-ricci-limit} combines Theorem \ref{thmE}  with Cheeger-Colding theory (and its latest developments \cite{CJN}, \cite{LS18}) to give a proof of Corollary \ref{coro-B}, see Theorem~\ref{thm metric completion} for its reformulation in the body of the text.

The final Section \ref{sec:GH} combines the two previous sections  and studies whether the compact metric spaces
$(X,d_t)$, where $d_t$ denotes the induced distance by the K\"ahler form $\omega_t$ solution to \eqref{eqn-krf-twist-shout}, converge, in the Gromov-Hausdorff topology, to the metric space $(X,d_T)$
as $t \rightarrow 0$: Proposition \ref{pro:equicont} first shows that the   spaces
$(X,d_t)$ are relatively compact in the Gromov-Hausdorff topology, a fundamental fact that lies at the heart of this paper. 
The first part of Theorem \ref{thmB} is established in Theorem \ref{thm:cvricci} where  the $L^1$ convergence of $\dot{\varphi}_t$ towards $\psi_+-\psi_{-}$ as $t \rightarrow 0$ is proved, a crucial property that implies the weak convergence of ${\rm Ric}(\omega_t)$ towards $\Ric(T)$ as $t \rightarrow 0$. Once the theory of compact surfaces with bounded integral curvature is briefly recalled in Section \ref{sec-alex-surf}, 
the aforementioned convergence property of $\Ric(\omega_t)$ is then invoked to prove Theorem \ref{thmD} through Theorem \ref{thm:fkr1dim}. 
Theorem \ref{thm:cvGHisolated} then proves the second part of Theorem \ref{thmB}, i.e. the Gromov-Hausdorff convergence of $(X,d_t)$ towards $(X,d_T)$ provided the singularities of the density of $T^n$ are furthermore isolated.

 \begin{ackn} 
 H.G would like to thank Shengxuan Zhou for fruitful discussions. 
 The authors are partially supported by the fondation Charles Defforey
and the Institut Universitaire de France.
 A.D. is partially supported by  ANR-24-CE40-0702 (Project OrbiScaR). 
 H.G. is partially supported by ANR-21-CE40-0010 (KARMAPOLIS). 
\end{ackn}

 \section{Preliminaries}  
 \label{sec-smoo-flow}
 
 In the whole article we let $(X,\omega_X)$ denote a compact K\"ahler manifold of complex dimension $n$.
 We set $d=\partial+\overline{\partial}$ and $d^c=\frac{1}{2 \pi i}(\partial-\overline{\partial})$ 
 so that
$dd^c =\frac{i}{\pi}\partial\overline{\partial}$.

 \subsection{Quasi-plurisubharmonic functions}

  \subsubsection{Regularization}
 
 A function is quasi-plurisub\-harmonic if it is locally given as the sum of  a smooth and a plurisubharmonic function.   
 Quasi-plurisubharmonic functions
$\f:X \rightarrow \R \cup \{-\infty\}$ satisfying
$
\omega_X+dd^c \f \geq 0
$
in the weak sense of currents are called $\omega_X$-psh functions.

\begin{defi}
We let $\PSH(X,\omega_X)$ denote the set of all $\omega_X$-plurisubharmonic functions which are not identically $-\infty$.  
\end{defi}

The set $\PSH(X,\omega_X)$ is a closed subset of $L^1(X)$ for the $L^1$-topology. 

\smallskip

A celebrated result of Demailly \cite{D92} ensures that any 
$\omega_X$-psh function is the decreasing limit of smooth $\omega_X$-psh functions.
To establish uniform estimates, one can thus approximate the initial data
$\f$ by a decreasing sequence of smooth $\omega_X$-psh functions $\f_{0,j}$
and establish a priori estimates on the smooth approximating flows $\f_{t,j}$.

 \subsubsection{Monge-Amp\`ere measures}
 
 In the whole article we let $T=\omega_X+dd^c \f$ be a positive closed current with {\it continuous}
 potential $\f \in \PSH(X,\omega_X)\cap C^0(X)$.
 It follows from Bedford-Taylor theory
 (see \cite{GZbook}) that the complex Monge-Amp\`ere measure 
 $T^n=(\omega_X+dd^c \f)^n$ is a well-defined positive Radon measure of
 total mass $V=\int_X \omega_X^n$.
 
 In this article we often assume that $T^n=f \omega_X^n$ is absolutely continuous
 with respect to a smooth volume form, with $f \in L^p(\omega_X^n)$ for some $p>1$.
 This ensures that $\f$ is H\"older continuous,
 as follows from the work of Kolodziej \cite{Kol08}.
  More generally we have the following weaker assumption.
 
  \begin{defi}
  A  positive  measure $\mu$ on $X$ satisfies Condition (K) if there exists 
    $dV_X$  a  smooth volume form, $f \geq 0$  a Lebesgue-measurable function,
$w: \R^+ \rightarrow  [1,+\infty)$ 
 a convex increasing weight 
 s.t. $\mu=f dV_X$ and $\int_X w \circ f \, dV_X <+\infty$,
 where
 $$
 w(t) \stackrel{+\infty}{\sim} t (\log t)^n (h \circ \log \circ \log t)^n 
 \; \; \text{ and } \; \;
 \int_0^{+\infty} \frac{dt}{h(t)}<+\infty.
 $$
  \end{defi}
  
  As shown by Kolodziej in \cite[Theorem 2.5.2]{Kol98},
  this is the optimal condition that ensures the continuity
  of Monge-Amp\`ere potentials.
     One can moreover obtain quantitative estimates
  on its modulus of continuity in terms of the growth of $w$ at infinity
  (see \cite{Kol08,DDGHKZ14,GPTW21,GGZ25}):
  
  \begin{thm} \label{thm:kolo}
  When $\mu$ satisfies Condition (K) and $\mu(X)=\int_X \omega_X^n$, 
  there exists a unique continuous $\omega_X$-psh function $\f$ such that
  $$
  (\omega_X+dd^c \f^n)=\mu=fdV_X,
  $$
  up to an additive constant.   When $f \in L^p(dV_X)$ for some $p>1$, then  $\f$ is H\"older continuous.
  \end{thm}

 \subsubsection{Diameter of K\"ahler metrics with H\"older potentials}\label{sec-diam-bd}
  
  Several uniform bounds for the diameter of K\"ahler manifolds have been provided in
  the past few years.
  We shall need the following estimate, which is a combination of 
  the main result of \cite{Kol08, DDGHKZ14} together with \cite[Theorem 4.1]{Li21}.

  \begin{thm} \label{thm:diameter}
  Let $\omega=\omega_X+dd^c \f$ be a K\"ahler form 
  and fix a smooth volume form $dV_X$.
  Assume that $\omega^n=f dV_X$ with
  $\int_X f^p dV_X\leq A$ for some $p>1,A>0$. Then for all $x,y \in X$,
  $$
  {\rm diam}(X,{\omega}) \leq C
  \; \; \text{ and } \; \;
d_{{\omega}} (x,y) \leq C d_{\omega_X}(x,y)^{\alpha}
  $$
  for some constants $C,\alpha>0$ that only depend on $X,\omega_X,dV_X,p$ and 
 an upper bound for  $\|f\|_{L^p}$.
  \end{thm}
  
  As shown in \cite{GPSS24,GGZ25} one can obtain a uniform bound on the diameter
  under slightly less restrictive integrability condition on the density $f$. However, Condition (K) is too weak to ensure finiteness of the diameter in general.
    On the other hand, Condition (K) coupled with a uniform Ricci lower bound on $\omega$,
  guarantees finiteness of the diameter as advocated by \cite{FGS20}
  (and pushed further in \cite[Theorem A]{GGZ25}).

 \subsection{Complex Monge-Amp\`ere flows}\label{cx-MA-flow}
 
 In the whole article we fix an initial datum $T=\omega_X+dd^c \f$ which is 
 a positive closed current cohomologous to $\omega_X$.
 We assume that the qpsh potential $\f$ is continuous and 
 $T^n=f \omega_X^n$, where $f \in L^p(\omega_X^n)$ for some $p>1$. It thus follows from 
 Theorem \ref{thm:kolo} that $\f$ is actually H\"older continuous.

 Since ${\rm Ric}(\omega_t)$ and ${\rm Ric}(\omega_X)$ both represent the first Chern class
$c_1(X)$  of $X$, it follows from the twisted K\"ahler-Ricci flow equation
  $$
\frac{\partial \omega_t}{\partial t}=-{\rm Ric}(\omega_t)+{\rm Ric}(\omega_X)
$$
 that the cohomology class of $\omega_t$ is constant, equal to  that
 of $\omega_X$. We can thus decompose $\omega_t=\omega_X+dd^c \f_t$ for some
 smooth $\omega_X$-psh function $\f_t(x)$ which is unique up to an additive normalizing constant
 $c(t)$. The flow equation therefore yields
 $$
 dd^c \dot{\f_t}=dd^c \log \frac{(\omega_X+dd^c \f_t)^n}{\omega_X^n},
 $$
 hence 
 $
 (\omega_X+dd^c \f_t)^n=e^{\dot{\f}_t+b(t)} \omega_X^n,
 $
 for some   $b(t) \in \R$. We normalize $\f$   so that $b \equiv 0$.
 
 \smallskip
 
 Solving the twisted K\"ahler-Ricci flow equation is thus equivalent to finding 
 a family of smooth $\omega_X$-psh functions $\f_t \in \PSH(X,\omega_X) \cap {\mathcal C}^{\infty}(X)$
 such that
 \begin{equation} \label{eq:CMAF}
 (\omega_X+dd^c \f_t)^n=e^{\dot{\f_t}} \omega_X^n \,\,
 \text{ on } X \times \R_+^*,
 \; \; \text{ and } \; \; 
 \f_t \underset{t\to 0}{\longrightarrow} \f
 \text{ in }  L^1(\omega_X^n).
 \end{equation}
 
 In this context the Cauchy problem for the complex Monge-Amp\`ere flow \eqref{eq:CMAF}
 admits a unique continuous solution:
 
 \begin{thm} \label{thm:GZ17}
 \cite{GZ17,DnL17,ST09}
 There exists a unique family $\f \in {\mathcal C}^0(X \times \R_+)$  such that,
 setting $\f_t=\f_{|X \times \{t\}}$,
 \begin{itemize}
 \item[$\bullet$] $\f\in {\mathcal C}^{\infty}(X \times \R_+^*)$ and 
 $\omega_t=\omega_X+dd^c \f_t$ is a K\"ahler form for any $t>0$;
 \item[$\bullet$] $\f$ is a solution to \eqref{eq:CMAF} on $X \times (0,+\infty)$;
 \item[$\bullet$] $\|\f_t-\f\|_{L^{\infty}(X)} \rightarrow 0$ as $t \rightarrow 0$.
 \end{itemize}
 \end{thm}
 
 Our aim in this article is to understand the metric convergence at $t=0$ of  $(X,d_{t})$,
 where $d_t$ denotes the Riemannian distance associated to $\omega_t$.

\section{A priori estimates along the flow} \label{sec:fkr}

We consider in this section the complex Monge-Amp\`ere flow
$$
(\omega_X+dd^c \f_t)^n=e^{\dot{\f_t}} \omega_X^n
$$
with initial condition $\f \in \PSH(X,\omega_X) \cap {\mathcal C}^0(X)$,
such that $(\omega_X+dd^c \f)^n=e^{\p^+-\p^-} \omega_X^n$.

Set $\omega_t:=\omega_X+dd^c \f_t$.
We approximate $\f, \p_{\pm}$ by decreasing sequences of smooth  $A \omega_X$-psh functions
and consider the approximating flows.  We establish uniform a priori estimates on the latter
and then pass to the limit.
In what follows we thus assume that all the quantities involved are smooth and we freely use the classical parabolic maximum principle.

\subsection{${\mathcal C}^0$-bounds}

The   maximum principle ensures that
$\f(t,x)$ is uniformly bounded:

\begin{prop} \label{pro:c0bounds}
For all $(x,t) \in X \times \R^+$, one has
$
\inf_X \f \leq \f_t(x) \leq \sup_X \f.
$
\end{prop}

If we further assume that there is a {\it convex}
 weight $w$ such that
$$
(\omega_X+dd^c \f)^n=f \omega_X^n
\; \; \text{ with } \; \;
I(0):=\int_X w \circ f \omega_X^n <+\infty,
$$
the latter control is   uniform along the flow as  follows from the following result.

\begin{lem} \label{lem:convexflot}
The function $t \mapsto  I(t)=\int_X w \circ f_t \, \omega_X^n$
is non-increasing along the flow
$$
{\dot{\f_t}}=\log f_t:= \log  \left[ \frac{(\omega_X+dd^c \f_t)^n}{\omega_X^n} \right].
$$
\end{lem}

This property has been established in \cite[Proposition 3.5]{GZ17}.
We include a proof for the reader's convenience.
We shall use it hereafter with $w(t)=t^p$, $p>1$.

\begin{proof}
Observe that
$\log f_t$ satisfies the following Heat-type equation,
$$
\frac{\partial f_t}{\partial t}=\Delta_{\omega_t} f_t-\frac{|\nabla_{\omega_t} f_t|^2}{f_t}.
$$
We infer
$
I'(t) = n\int_X \frac{w' \circ f_t}{f_t} \, dd^c f_t \wedge \omega_t^{n-1}
-n\int_X \frac{w' \circ f_t}{f_t^2} \, df_t \wedge d^c f_t \wedge \omega_t^{n-1}.
$
Integrating by parts yields
\begin{eqnarray*}
\lefteqn{  \! \! \! \! \! \! \!  \! \! \! \! \! \! \! 
\int_X \frac{w' \circ f_t}{f_t} \, dd^c f_t \wedge \omega_t^{n-1}
= -\int d \left( \frac{w' \circ f_t}{f_t} \right) \wedge d^c f_t \wedge \omega_t^{n-1}} \\
&=& -\int_X \frac{w'' \circ f_t}{f_t} \, df_t \wedge d^c f_t \wedge \omega_t^{n-1}
+\int_X \frac{w' \circ f_t}{f_t^2} \, df_t \wedge d^c f_t \wedge \omega_t^{n-1},
\end{eqnarray*}
therefore
$
I'(t) = -n \int_X \frac{w'' \circ f_t}{f_t} \, df_t \wedge d^c f_t \wedge \omega_t^{n-1} \leq 0,
$
as claimed.
\end{proof}

\begin{cor} \label{cor:equicont1}
Assume $\int_X f^p  \omega_X^n <+\infty$ for some $p>1$. 
Then $\exists \alpha, C>0$ such that
\[ \forall  x,y \in X, \,\forall t>0, \quad d_{t} (x,y) \leq C d_{\omega_X}(x,y)^{\alpha}.\]
\end{cor}
   
   \begin{proof}
   Lemma \ref{lem:convexflot} ensures that $\omega_t^n=f_t \omega_X^n$ with
   $\int_X f_t^p \omega_X^n \leq \int_X f^p \omega_X^n <+\infty$.
   The result now follows from Theorem~\ref{thm:diameter}.
   \end{proof}

  \subsection{Bounds on $\dot{\f}_t$}

 Recall that we assume that the initial datum satisfies
$$
(\omega_X+dd^c \f)^n=f \omega_X^n =e^{\p_+-\p_-} \omega_X^n,
$$
where $\p_{\pm}$ are qpsh functions.

\begin{lem} \label{lem:bddphidot}
There exists $C>0$ such that for all $0 \leq t \leq 1$ and for all $x \in X$,
$$
 {\p_+}(x) -C   \leq   \dot{\f_t}(x) \leq C-{\p_-}(x) .
$$
\end{lem}

\begin{proof}
We fix $A>0$  so that $\p_{\pm} \in \PSH(X,A \omega_X)$ and
consider 
$$
H(t,x)=\dot{\f_t}(x)+\p_-(x)-A \f_t(x).
$$
Since $\f_t(x)$ is uniformly bounded (by Proposition \ref{pro:c0bounds}), the bound from above on
$\dot{\f_t}(x)$ is equivalent to a uniform bound from above for $H$.
If the maximum of $H$ is attained along $(t=0)$ we obtain
$$
H_{\rm max} \leq \sup_X \p_+-A \inf \f_t(x) \leq C,
$$
since $\dot{\f_0}(x)=\p_+(x)-\p_-(x)\leq \sup_X \p_+-\p_-(x)$.

Assume now that the maximum is attained at some point
$({t_0},x_0)$ with ${t_0}>0$. At this point we have
$0 \leq \left(\frac{\partial}{\partial t}-\Delta_t \right)(H)$,
where we let $\Delta_t$ denote the Laplacian with respect to $\omega_t:=\omega_X+dd^c \f_t$.
A direct computation shows that
$$
\left(\frac{\partial}{\partial t}-\Delta_t \right)(H)
=A[n-\dot{\f_t}(x)]-{\rm tr}_{\omega_t}(A \omega+dd^c \p_-)
\leq A[n-\dot{\f_t}(x)].
$$
Thus $\dot{\f_{{t_0}}}(x_0) \leq n$. Since $\p_-(x_0) \leq \sup_X \p_-$ and 
$|\f_{t}(x) | \leq C$, we infer again
$$
H_{\rm max} \leq n +\sup_X \p_- +A C\leq C'.\\
$$

We now establish the uniform bound from below on $\dot{\f_t}(x)$.
Consider 
$$
G(t,x)=\dot{\f_t}(x)-\p_+(x)+(A+1) \f_t(x).
$$
The function $G$ is bounded from below along $(t=0)$ since
$\dot{\f_0}(x) -\p_+(x) \geq -\sup_X \p_-$
and $\f$ is bounded.  A direct computation shows that
$$
\left(\frac{\partial}{\partial t}-\Delta_t \right)(G)
\geq (A+1) [\dot{\f_t}(x)-n]+{\rm tr}_{\omega_t}(\omega)
\geq (A+1) [\dot{\f_t}(x)-n]+n e^{-\frac{1}{n} \dot{\f_t}(x)},
$$
since
$$
{\rm tr}_{\omega_t}(\omega) \geq n \left( \frac{\omega_t^n}{\omega_X^n} \right)^{1/n} =n e^{-\frac{1}{n} \dot{\f_t}(x)}.
$$
If $G$ reaches its minimum at a point $({t_0},x_0)$ with ${t_0}>0$, we infer
$$
0 \geq (A+1) [\dot{\f_{{t_0}}}(x_0)-n]+n e^{-\frac{1}{n} \dot{\f_{{t_0}}}(x_0)},
$$
hence $\dot{\f_{{t_0}}}(x_0) \geq -C(A,n)$. The conclusion follows.
\end{proof}

 \subsection{Laplacian bounds}
 
The goal of this section is to establish the following result.

\begin{thm} \label{thm:laplace}
There exists $C>0$ such that for all $0 \leq t \leq 1$ and for all $x \in X$,
$$
\frac{1}{C} e^{\p_+} \omega_X \leq \omega_X+dd^c \f_t \leq C e^{-\p_-} \omega_X.
$$
\end{thm}

 \subsubsection{Parabolic Chern-Lu inequality}

 We are going to use  the following inequality due to Chern-Lu \cite{Chern, Lu}.

\begin{lem} \label{lem:chernlu}
 Let $(X,\om_X)$ and $(Y, \om_Y)$ be two Kähler manifolds, and let $f:X\to Y$ be a holomorphic map such that $\partial f$ does not vanish. Then 
\[  \Delta_{\omega_X} \log |\partial f|^2 \ge \frac{(\Ric \om_X)^{\sharp} \otimes \om_Y(\partial f, \overline {\partial f})}{|\partial f|^2}-\frac{ \om_X^{\sharp}\otimes \om_X^{\sharp} \otimes R^Y(\partial  f, \overline {\partial  f},\partial  f, \overline {\partial  f})}{|\partial f|^2} \]
where $\partial  f$ is viewed as a section of $T^*_X \otimes f^*T_Y$, $(\Ric \om_X)^{\sharp}$ (resp. $\om_X^{\sharp}$) is the hermitian form induced on $T_X^*$ by $\Ric \om_X$ (resp. $\om_X$) via $\om_X$ and $R^Y$ is the Chern curvature form of $\om_Y$.
\end{lem}

We apply the latter lemma to a manifold $X$ endowed with two Kähler metrics $\om, \beta$ and $f=\mathrm{Id}:(X,\beta)\to (X,\om_X)$. We choose a system of local holomorphic coordinates $(z_1, \ldots, z_n)$ on $X$ such that the hermitian matrices associated to $\om_X$ (resp. $\beta$) are $(h_{i\bar j})$ (resp. $(g_{i\bar j})$). The associated curvature tensors are denoted respectively by $R_{i\bar j k \bar \ell}^{\om_X}$ and $R_{i\bar j k \bar \ell}^{\beta}$. We also introduce the Ricci curvature tensor for the Kähler metric $\beta$ as $R_{i\bar j}^{\beta}=g^{k \bar \ell} R^{\beta}_{i \bar j k \bar{ \ell}}$.  The hermitian form $(\Ric \beta)^{\sharp}$ gives rise to the tensor $R_{i\bar j}^{\beta \sharp}=g^{i \bar k}g^{\ell \bar j} R^{\beta}_{k \bar \ell}$. Finally, one has $\partial f = dz_i \otimes \frac{\partial}{\partial z_i}$ and $|\partial f|^2=\tr_{\beta}\om_X$.
In coordinates, Chern-Lu formula now becomes
\[
\Delta_{\beta} \log \tr_{\beta}\om_X \ge \frac{1}{\tr_\beta \om_X} \big( g^{i \bar k} R^{\beta}_{k \bar \ell} g^{\ell \bar j}h_{i\bar j}- g^{i \bar j}g^{k\bar \ell} R^{\om_X}_{i \bar j k \bar \ell} \big)
\]

Assume that there exists an upper bound $B>0$ on the holomorphic bisectional curvature of $(X,\omega_X)$ so that 
$
R^{\om_X}_{i\bar j k \bar \ell} \le B\, (h_{i\bar j}h_{k \bar \ell}+h_{i\bar \ell}h_{k \bar j}).
$
 As $g^{i \bar j}g^{k\bar \ell}(h_{i\bar j}h_{k \bar \ell}+h_{i\bar \ell}h_{k \bar j}) \le 2 (\tr_{\beta}\om_X)^2$, 
 we get 
\[
\Delta_{\beta} \log \tr_{\beta}\om_X \ge \frac{1}{\tr_\beta \om_X}  g^{i \bar k} R^{\beta}_{k \bar \ell} g^{\ell \bar j}h_{i\bar j}
-2B \tr_{\beta}\om_X
\]

If the forms $\beta,\omega_X$ depend on a time parameter $t$, so that
\[
(\Ric +\frac{\partial}{\partial t})\beta \ge - A \beta-(C_1\om_X+dd^c \psi), \qquad \frac{\partial}{\partial t}\om_X \le C_2 \om_X \quad \mbox{and} \quad \frac{\partial}{\partial t}\psi \le C_3
\]
for some constants $A,C_1,C_2,C_3>0$ and some function $\psi \in \PSH(X,C_1\om_X)$. In particular, we get
\[ g^{i \bar k} R^{\beta}_{k \bar \ell} g^{\ell \bar j}h_{i\bar j} \ge -A\tr_\beta \om_X - \tr_\beta \om_X \cdot \tr_\beta(C_1\om_X+dd^c\psi)-\tr(g^{-1}\frac{\partial}{\partial t}gg^{-1}h)\]
Combining that inequality with the identity below 
\[\frac{\partial}{\partial t}(\log \tr_{\beta}\om_X)=\frac{1}{\tr_\beta \om_X}\cdot (-\tr(g^{-1}\frac{\partial}{\partial t}gg^{-1}h)+\tr(g^{-1}\frac{\partial}{\partial t}h))\]
we get
\[ (\Delta_{\beta}-\frac{\partial}{\partial t}) \log \tr_{\beta}\om_X \ge-A-\tr_\beta(C_1\om_X+dd^c \psi) -C_2-2B\tr_{\beta}\om_X \]
As 
$\Delta_\beta \psi =\tr_\beta(C_1\om_X+dd^c\psi)-C_1\tr_\beta \om_X$
we infer 
\[ (\Delta_{\beta}-\frac{\partial}{\partial t})( \log \tr_{\beta}\om_X+\psi) \ge-A-(2B+C_1)\tr_\beta\om_X -C_2-C_3 \]

Finally, assume that $\omega_X$ and $\beta$ are cohomologous, i.e. that there exists a function $\varphi$ such that $\beta=\om_X+dd^c \varphi$. Then, 
$\Delta_\beta(-\varphi)=\tr_\beta \om_X - n$
and therefore  
\[ (\Delta_{\beta}-\frac{\partial}{\partial t})( \log \tr_{\beta}\om_X+\psi-C_4\varphi) \ge-A+\tr_\beta\om_X -C_4(n-\frac{\partial}{\partial t}\varphi)-C_2-C_3 \]
if we set $C_4:=2B+C_1+1$.
In the end, we have proved

\begin{lem} \label{lem:chernlupara}
 Let $X$ be a Kähler manifold and let $\omega_X=\omega_X(t),\beta=\beta(t)$ be two Kähler forms depending smoothly on a parameter $t>0$. Assume that
 \begin{enumerate}
 \item $\om_X$ and $\beta$ are cohomologous, i.e. there exists a function $\varphi$ such that $\beta=\om_X+dd^c \varphi$.
 \item The bisectional curvature of $\om_X$ is bounded above by a constant $B$; i.e. 
 $$i\Theta(T_X,\om_X) \le B \, \om_X \otimes \mathrm{Id}_{T_X}.$$
 \item 
$ (\Ric +\frac{\partial}{\partial t})\beta \ge - A \beta-(C_1\om_X+dd^c \psi), \; \frac{\partial}{\partial t}\om_X \le C_2 \om_X \; \mbox{and} \quad \frac{\partial}{\partial t}\psi \le C_3$,
for some constants $A,C_1,C_2,C_3>0$ and some function $\psi \in \PSH(X,C_1\om_X)$.
 \end{enumerate}
 Then, there exist two constants $K,M>0$ depending only on $n,A,B,C_1,C_2,C_3$ such that
 \[ (\Delta_{\beta}-\frac{\partial}{\partial t})( \log \tr_{\beta}\om_X+\psi-K\varphi) \ge\tr_\beta\om_X +K \cdot \frac{\partial}{\partial t}\varphi-M.\]
 \end{lem}

We shall also use the following classical Laplacian estimates
which goes back to the celebrated ${\mathcal C}^2$-estimate of Yau \cite{Yau78}
(the above version is due to Siu \cite{Siu87}). 

\begin{lem} \label{lem:siu}
Let $\omega_X,\beta$ be K\"ahler forms. Then
$$
-\Delta_{\beta} \log {\rm tr}_{\omega_X}(\beta) \leq 
\frac{{\rm tr}_{\omega_X}(\Ric(\beta))}{{\rm tr}_{\omega_X}(\beta)}+B {\rm tr}_{\beta}(\omega_X),
$$
where $-B$ is a lower bound on the holomorphic bisectional curvature of $(X,\omega_X)$.
\end{lem}

 \subsubsection{Proof of Theorem \ref{thm:laplace}}
 We proceed in two steps. 
 
 \medskip

{\it We first take care of the upper bound}. 
It is a parabolic version of the main result of \cite{Pau08}.
Consider 
$$
H(t,x)=\log {\rm tr}_{\omega_X}(\omega_t)-M \f_t+\p_-,
$$
where $M>0$ is chosen hereafter.
The result follows if we can uniformly bound $H$ from above, since $\f_t$ is uniformly bounded
(by Proposition \ref{pro:c0bounds}).
If the maximum of $H$ is reached along $(t=0)$, then
we are done by the main result of \cite{Pau08}.

We can thus assume that the maximum of $H$ is reached at $(t_0,x_0)$ with $t_0>0$.
We let $\Delta_t$ denote the Laplacian with respect to $\omega_t:=\omega_X+dd^c \f_t$.
Observe that 
$\Ric \omega_t =\Ric \omega_X-dd^c \dot{\f_t}$ hence
$$
{\rm tr}_{\omega_X}(\Ric \omega_t)={\rm tr}_{\omega}(\Ric \omega_X)-\Delta_{\omega_X}(\dot{\f_t})
$$
while
$$
\frac{\partial}{\partial t} \log {\rm tr}_{\omega_X}(\omega_t)=
\frac{\Delta_{\omega_X}(\dot{\f_t})}{{\rm tr}_{\omega_X}(\omega_t)}.
$$
It follows therefore from Lemma \ref{lem:siu} applied to $\beta=\omega_t$ that
\begin{eqnarray*}
\left(\frac{\partial}{\partial t}-\Delta_t \right)(\log {\rm tr}_{\omega_X}(\omega_t))
&\leq & \frac{{\rm tr}_{\omega_X}(\Ric(\omega_X))}{{\rm tr}_{\omega_X}(\omega_t)}+B {\rm tr}_{\omega_t}(\omega_X) \\
&\leq & \frac{C}{{\rm tr}_{\omega_X}(\omega_t)}+B {\rm tr}_{\omega_t}(\omega_X),
\end{eqnarray*}
hence at $(t_0,x_0)$,
\begin{eqnarray*}
0 \leq \left(\frac{\partial}{\partial t}-\Delta_t \right)(H) 
& \leq & \frac{C}{{\rm tr}_{\omega_X}(\omega_t)}+B {\rm tr}_{\omega_t}(\omega_X) -M \dot{\f_t}+M\Delta_t \f_t-\Delta_t \p_-\\
&\leq &  \frac{C}{{\rm tr}_{\omega_X}(\omega_t)}+(A+B-M) {\rm tr}_{\omega_t}(\omega_X)-M \dot{\f_t}+Mn.
\end{eqnarray*}
We choose $M=A+B+1$ and note that we can assume wlog ${\rm tr}_{\omega_X}(\omega_t) \geq 1$ at the point $(t_0, x_0)$ to obtain
\begin{eqnarray} \label{eq:laplace1}
{\rm tr}_{\omega_t}(\omega_X) \leq C'-M \dot{\f_t}.
\end{eqnarray}

The end of the proof in the case $n=1$ is simpler to handle, so we focus on the case $n \geq 2$ form now on. Recall that
$$
{\rm tr}_{\omega_X}(\omega_t) \leq n \left( \frac{\omega_t^n}{\omega_X^n} \right) \left[{\rm tr}_{\omega_t}(\omega_X) \right|^{n-1}
\Longrightarrow 
c_n e^{-\frac{\dot{\f_t}}{n-1}} \left[ {\rm tr}_{\omega_X}(\omega_t) \right]^{\frac{1}{n-1}}
\leq {\rm tr}_{\omega_t}(\omega_X).
$$
It follows therefore from \eqref{eq:laplace1} that
\begin{equation}
\label{ineq tr}
c_n  \left[ {\rm tr}_{\omega_X}(\omega_t) \right]^{\frac{1}{n-1}} \leq C'e^{\frac{\dot{\f_t}}{n-1}}+M (-\dot{\f_t})e^{\frac{\dot{\f_t}}{n-1}}.
\end{equation}
We deal with two cases separately. 

\medskip

$\bullet$ If $\dot\f_{t_0}(x_0)\le 0$, then 
\[{\rm tr}_{\omega_X}(\omega_{t_0}) (x_0) \leq C_4,\]
thanks to \eqref{ineq tr} and the fact that $\sup_{y \geq 0} y \exp(-\frac{y}{n-1}) =\alpha_n <+\infty$. As a result, we have
$$
H_{\rm max}=\log {\rm tr}_{\omega_X}(\omega_{t_0}) (x_0)+\p_-(x_0)-M \f_{t_0}(x_0) \leq\log C_4+\sup_X \p_-+C_0
$$
and the proof of the upper bound is complete in this first case. 

\medskip

$\bullet$ If $\dot\f_{t_0}(x_0)\ge 0$, then \eqref{ineq tr} shows 
\[{\rm tr}_{\omega_X}(\omega_t) e^{-\dot{\f_t}} \leq C''\quad \mbox{ at} \,\, (t_0, x_0).\]
Since $\p_- \leq -\dot{\f_t}+C$ by Lemma~\ref{lem:bddphidot}, this yields
\begin{eqnarray*}
H_{\rm max}&=& \log {\rm tr}_{\omega_X}(\omega_{t})(t_0,x_0)+\psi_{-}(x_0)+O(1) \\
&\leq &  \log e^{-\dot{\f_t}} {\rm tr}_{\omega_X}(\omega_t)(t_0, x_0) +O(1)+C \leq C_3.
\end{eqnarray*}
and the proof of the upper bound is now fully complete.

\bigskip

{\it We now take care of the lower bound}. From the equation, we see that 
$$
\Ric \om_t =\Ric \om_X-dd^c \dot \f_t
= \Ric \om_X - \partial_t \om_t.
$$
Therefore, one can apply Lemma~\ref{lem:chernlupara} with $\beta=\om_t$,  $\psi=0, A=C_2=C_3=0$ in order to get two constants $K_1,M_1>0$ satisfying
\[(\partial_t-\Delta_t)(\log \tr_{\om_t}\om_X-K_1\varphi_t) \le -\tr_{\om_t}\om_X -K_1 \dot \varphi_t+M_1. \]
If one introduces the quantity 
$$
\wt H(t,x):=\log \tr_{\om_t}(\omega_X)-K \f_t+\psi_+,
$$
where $K:=K_1+A$, then one infers immediately that 
\[(\partial_t-\Delta_t) \wt H \le -\tr_{\om_t}\om_X -K \dot \varphi_t+M \]
where $M:=M_1+nA$. 

If the maximum of $\wt H$ is reached at a point $(t_0,x_0)$ where $t_0=0$, then the result follows from \cite[Corollary~2.3]{GSS}. 
We can thus assume that the maximum of $H$ is reached at $(t_0,x_0)$ with $t_0>0$. 
At that point, one has $(\partial_t-\Delta_t) \wt H \ge 0$ and therefore
\[\tr_{\om_t}\om_X \le M-K \dot \varphi_t.\]
We distinguish two cases. 

\smallskip

$\bullet$ If $\dot\f_{t_0}(x_0) \ge 0$, then $\tr_{\om_{t_0}}\om_X(x_0) \le M$ hence
\[
\wt H(t_0, x_0) \le M':=M+K\sup |\f_t| +\sup_X \psi_+.
\]
We infer
$
\log \tr_{\om_t}\om_X+\psi_+ = \wt H+K\f_t 
 \le M'+K\sup |\f_t|
$
and we are done. 

\smallskip

$\bullet$ We now assume that $\dot\f_{t_0}(x_0) < 0$. Ultimately, we want to show $\wt H \le O(1)$, 
or equivalently $\wt H(t_0,x_0) \le O(1)$. Let us introduce a universal constant $C=C(K,M)$ such that 
$x+\log(M-Kx) \le C$ for any $x\in (-\infty,0]$. Using Lemma~\ref{lem:bddphidot}, one gets 
\begin{align*}
\wt H(t_0,x_0) & \le \log \tr_{\om_{t_0}}\om_X(x_0)+\psi_+(x_0) +O(1)\\
&\le  \log \tr_{\om_{t_0}}\om_X(x_0)+\dot\f_{t_0}(x_0)+O(1)\\
&\le \log(M-K\dot\f_{t_0}(x_0))+\dot\f_{t_0}(x_0)+O(1)
\le C+O(1)
\end{align*}
and the result follows.
$\Box$

\subsection{Higher order estimates}

\subsubsection{Positive times}

Using the complex parabolic Evans-Krylov theory together with Schauder’s estimates, it follows from our previous estimates that the following higher order a priori estimates hold:

\begin{prop} \label{pro:higher}
Fix $0 <\e <T<+\infty$ and $k \in \N$. There exists 
$C(k,\e,T)>0$ such that 
$$
\|\f\|_{{\mathcal C}^k(X \times [\e,T])} \leq C(k,\e,T).
$$
\end{prop}

\subsubsection{Regularity up to time zero}

If the functions $\p_{\pm}$ are smooth in some Zariski open set $\Omega \subset X$,
one can further obtain uniform estimates on compact subsets $K$ of $\Omega$,
$$
\|\f\|_{{\mathcal C}^k(K \times [0,T])} \leq C(k,K,T).
$$
Since $\f_t$ converges to $\f$ in ${\mathcal C}^0(X)$, we infer the following result.

\begin{prop} \label{pro:higherbis}
If the functions $\p_{\pm}$ are smooth in some Zariski open set $\Omega \subset X$,
then $\f_t$ converges to $\f$ in ${\mathcal C}^{\infty}_{\rm loc}(\Omega)$ as $t \rightarrow 0$.
\end{prop}

 When $\p_+=0$ one can reinforce some of these estimates as follows.

\begin{prop} \label{pro:quasiconcave}
If  $\p_+=0$  then there exists $C>0$ such that $\f_t-Ct^2$ is concave.
\end{prop}

\begin{proof}
Recall that $\ddot{\f}_t=\Delta_t (\dot{\f}_t)$ and observe that
\begin{eqnarray*}
\left( \frac{\partial}{\partial t}-\Delta_t \right)(\ddot{\f}_t)
&=& n(n-1) \frac{(dd^c \dot{\f}_t)^2 \wedge \omega_t^{n-2}}{\omega_t^n}
-n^2 \left(\frac{dd^c \dot{\f}_t \wedge \omega_t^{n-1}}{\omega_t^n} \right)^2 \\
&\leq & -n \left(\frac{dd^c \dot{\f}_t \wedge \omega_t^{n-1}}{\omega_t^n} \right)^2 \leq 0.
\end{eqnarray*}
Thus $\ddot{\f}_t$ reaches its maximum along $(t=0)$ by the maximum principle. Now
$$
\ddot{\f}_0=\Delta_0 (\dot{\f}_0)=\Delta_{\omega_0} (-\p_-) \leq A {\rm tr}_{\omega_0}(\omega_X) \leq C,
$$
as follows from Theorem \ref{thm:laplace}.
\end{proof}

\subsubsection{Geometric singularities} \label{sec:alix}

We assume here that $\p_+=0$ and $|\partial \p_-|_{\omega_X} \leq e^{-\alpha \p_-}$
for some exponent $\alpha>0$, and establish  much more precise higher order information.
 
 \begin{prop}\label{pro:alix1}

Assume $\psi_+=0$. If $|\partial\psi_{-}|_{\omega_X}\leq e^{-\alpha\psi_{-}}$ for some $\alpha> 0$,  then there exist constants $C$ and $\gamma>0$ such that for $t>0$,
\begin{equation*}
|\nabla^{g(t)}\dot{\f_t}|_{g(t)}\leq Ce^{-\gamma\psi_{-}+Ct}.
\end{equation*}
\end{prop}

\begin{rem}
Note that the gradient of $\dot{\f_t}$ does not blow up in time as $t$ tends to $0^+$.
\end{rem}

\begin{proof}
Denote by $g(t)$ and $g_X$ respectively the Riemannian metrics induced by the K\"ahler forms $\omega(t)$ and $\omega_X$.
For $\gamma_0>0$, compute as follows:
\begin{equation}
\begin{split}\label{it-never-ends}
\left(\partial_t-\frac{1}{2}\Delta_{g(t)}\right)\left(e^{\gamma_0\psi_{-}}\dot{\f_t}^2\right)&=\left(-\frac{1}{2}\Delta_{g(t)}e^{\gamma_0\psi_{-}}\right)\dot{\f_t}^2-\gamma_0g(t)\left(\nabla^{g(t)}\psi_{-},\nabla^{g(t)}\dot{\f_t}^2\right)e^{\gamma_0\psi_{-}}\\
&\quad-\,e^{\gamma_0\psi_{-}}|\nabla^{g(t)}\dot{\f_t}|^2_{g(t)}.
\end{split}
\end{equation}
Now, as $i\partial\overline{\partial}\psi_{-}+A\omega_X\geq 0$ for some $A> 0$, 
\begin{equation*}
\begin{split}
-\Delta_{g(t)}e^{\gamma_0\psi_{-}}&=\left(-\gamma_0\Delta_{g(t)} \psi_{-}-\gamma_0^2|\nabla^{g(t)}\psi_{-}|^2_{g(t)}\right)e^{\gamma_0\psi_{-}}\leq C\gamma_0 Ae^{\gamma_0\psi_{-}},
\end{split}
\end{equation*}
where $C$ is a positive constant that may vary from line to line but which is independent of $t>0$. Here we have invoked Theorem \ref{thm:laplace} in the last inequality. 

Therefore, Young's inequality applied to the mixed term of \eqref{it-never-ends} gives us:
\begin{equation}
\begin{split}\label{est-0-undergrad}
\left(\partial_t-\frac{1}{2}\Delta_{g(t)}\right)\left(e^{\gamma_0\psi_{-}}\dot{\f_t}^2\right)&
\leq C\left(\gamma_0A+\gamma_0^2|\nabla^{g(t)}\psi_{-}|^2_{g(t)}\right)e^{\gamma_0\psi_{-}}\dot{\f_t}^2 -\frac{1}{2}e^{\gamma_0\psi_{-}}|\nabla^{g(t)}\dot{\f_t}|^2_{g(t)}\\
&\leq C -\frac{1}{2}e^{\gamma_0\psi_{-}}|\nabla^{g(t)}\dot{\f_t}|^2_{g(t)} ,
\end{split}
\end{equation}
where $C$ may vary from line to line. Here we have used  $|\nabla^{g(t)}\psi_{-}|_{g(t)}\leq C|\nabla^{g_X}\psi_{-}|_{g_X}$ according to Theorem \ref{thm:laplace} and the fact that $\psi_+=0$ in the last line: indeed  $\dot{\f_t}^2\leq C(\psi_{-})^2$ according to Lemma \ref{lem:bddphidot} together with the geometric assumption on the $g_X$-gradient of $\psi_{-}$, one concludes as above by choosing $\gamma_0>0$ large enough compared to $\alpha$. 

\begin{claim}\label{claim-bochner}
There exists $\gamma_0>0$ such that for $\gamma_1>\gamma_0$, there exists $C>0$ such that,
\begin{equation}
\begin{split}\label{est-I-grad}
\left(\partial_t-\frac{1}{2}\Delta_{g(t)}\right)\left(e^{\gamma_1\psi_{-}}|\nabla^{g(t)}\dot{\f_t}|_{g(t)}^2\right)&\leq Ce^{\gamma_0\psi_{-}}|\nabla^{g(t)}\dot{\f_t}|_{g(t)}^2.
\end{split}
\end{equation}
\end{claim}

\begin{proof}[Proof of Claim \ref{claim-bochner}]
Observe that by a straightforward computation:
{\small
\begin{equation*}
\begin{split}
&\left(\partial_t-\frac{1}{2}\Delta_{g(t)}\right)\left(e^{\gamma_1\psi_{-}}|\nabla^{g(t)}\dot{\f_t}|_{g(t)}^2\right)\\
&=\left(\left(\partial_t-\frac{1}{2}\Delta_{g(t)}\right)e^{\gamma_1\psi_{-}}\right)|\nabla^{g(t)}\dot{\f_t}|_{g(t)}^2-g(t)\left(\nabla^{g(t)}e^{\gamma_1\psi_{-}},\nabla^{g(t)}|\nabla^{g(t)}\dot{\f_t}|_{g(t)}^2\right)\\
&\quad+e^{\gamma_1\psi_{-}}\left(\partial_t-\frac{1}{2}\Delta_{g(t)}\right) |\nabla^{g(t)}\dot{\f_t}|_{g(t)}^2\\
&=\left(-\frac{\gamma_1}{2}\Delta_{g(t)}\psi_{-}-\gamma_1^2|\nabla^{g(t)}\psi_{-}|^2_{g(t)}\right)e^{\gamma_1\psi_{-}}|\nabla^{g(t)}\dot{\f_t}|_{g(t)}^2\\
&\quad-\gamma_1e^{\gamma_1\psi_{-}}g(t)\left(\nabla^{g(t)}\psi_{-},\nabla^{g(t)}|\nabla^{g(t)}\dot{\f_t}|_{g(t)}^2\right)+e^{\gamma_1\psi_{-}}\left(\partial_t-\frac{1}{2}\Delta_{g(t)}\right) |\nabla^{g(t)}\dot{\f_t}|_{g(t)}^2.
\end{split}
\end{equation*}
}
Therefore, Young's inequality implies,
{\small
\begin{equation}
\begin{split}\label{lovely-formula}
&\left(\partial_t-\frac{1}{2}\Delta_{g(t)}\right)\left(e^{\gamma_1\psi_{-}}|\nabla^{g(t)}\dot{\f_t}|_{g(t)}^2\right)\leq\left(C\gamma_1 A-\frac{\gamma_1^2}{2}|\nabla^{g(t)}\psi_{-}|^2_{g(t)}\right)e^{\gamma_1\psi_{-}}|\nabla^{g(t)}\dot{\f_t}|_{g(t)}^2\\
&\quad+2\gamma_1e^{\gamma_1\psi_{-}}|\nabla^{g(t)}\psi_{-}|_{g(t)}|\nabla^{g(t)}\dot{\f_t}|_{g(t)}|\nabla^{g(t),\,2}\dot{\f_t}|_{g(t)}+e^{\gamma_1\psi_{-}}\left(\partial_t-\frac{1}{2}\Delta_{g(t)}\right) |\nabla^{g(t)}\dot{\f_t}|_{g(t)}^2\\
&\leq C\left(A+|\nabla^{g(t)}\psi_{-}|^2_{g(t)}\right)e^{\gamma_1\psi_{-}}|\nabla^{g(t)}\dot{\f_t}|_{g(t)}^2\\
&\quad +e^{\gamma_1\psi_{-}}|\nabla^{g(t),\,2}\dot{\f_t}|^2_{g(t)}+e^{\gamma_1\psi_{-}}\left(\partial_t-\frac{1}{2}\Delta_{g(t)}\right) |\nabla^{g(t)}\dot{\f_t}|_{g(t)}^2,
\end{split}
\end{equation}
}
where $C$ is a positive constant depending on $\gamma_1$ that may vary from line to line.

Now, the Bochner formula together with the fact that $\partial_tg(t)=-\Ric(g(t))+\Ric(g_X)$ allow us to write:
\begin{equation*}
\begin{split}
\left(\partial_t-\frac{1}{2}\Delta_{g(t)}\right)|\nabla^{g(t)}\dot{\f_t}|^2_{g(t)}&=-|\nabla^{g(t),\,2}\dot{\f_t}|^2_{g(t)}-\Ric(g_X)(\nabla^{g(t)}\dot{\f_t},\nabla^{g(t)}\dot{\f_t})\\
&\leq -|\nabla^{g(t),\,2}\dot{\f_t}|^2_{g(t)}+C|\nabla^{g(t)}\dot{\f_t}|^2_{g_X}\\
&\leq -|\nabla^{g(t),\,2}\dot{\f_t}|^2_{g(t)}+C|\nabla^{g(t)}\dot{\f_t}|^2_{g(t)},
\end{split}
\end{equation*}
where $C$ is a positive constant that may vary from line to line. Here we have used Theorem \ref{thm:laplace} in the last line.

Finally, the combination of \eqref{lovely-formula} and the previous estimate gives: 
\begin{equation*}
\begin{split}
\left(\partial_t-\frac{1}{2}\Delta_{g(t)}\right)\left(e^{\gamma_1\psi_{-}}|\nabla^{g(t)}\dot{\f_t}|_{g(t)}^2\right)&\leq C\left(A+|\nabla^{g(t)}\psi_{-}|^2_{g(t)}\right)e^{\gamma_1\psi_{-}}|\nabla^{g(t)}\dot{\f_t}|_{g(t)}^2\\
&\leq Ce^{\gamma_0\psi_{-}}|\nabla^{g(t)}\dot{\f_t}|_{g(t)}^2,
\end{split}
\end{equation*}
since by Theorem \ref{thm:laplace} together with the geometric assumption on the $g_X$-gradient of $\psi_{-}$, 
$$
|\nabla^{g(t)}\psi_{-}|^2_{g(t)}e^{(\gamma_1-\gamma_0)\psi_{-}}
\leq C|\nabla^{g_X}\psi_{-}|_{g_X}^2e^{(\gamma_1-\gamma_0)\psi_{-}}\leq C.
$$
This ends the proof of the claim.
\end{proof}

Let us sum the estimates  \eqref{est-0-undergrad} and [\eqref{est-I-grad}, Claim \ref{claim-bochner}] to get for $B>0$ large enough:
\begin{equation*}
\begin{split}\label{est-I-grad-bis}
&\left(\partial_t-\frac{1}{2}\Delta_{g(t)}\right)\left(B\,e^{\gamma_0\psi_{-}}\dot{\f_t}^2+e^{\gamma_1\psi_{-}}|\nabla^{g(t)}\dot{\f_t}|_{g(t)}^2\right)\\
&\leq BC -\frac{B}{2}e^{\gamma_0\psi_{-}}|\nabla^{g(t)}\dot{\f_t}|^2_{g(t)} +Ce^{\gamma_0\psi_{-}}|\nabla^{g(t)}\dot{\f_t}|_{g(t)}^2
\leq BC.
\end{split}
\end{equation*}

 Assuming $\gamma_1>0$ is chosen such that  
\begin{eqnarray*}
 \limsup_{t\rightarrow 0^+}e^{\gamma_1\psi_{-}}|\nabla^{g(t)}\dot{\f_t}|_{g(t)}^2
 &\leq & \limsup_{t\rightarrow 0^+}e^{\gamma_1\psi_{-}}|\nabla^{g_X}\dot{\f_t}|_{g_X}^2 \\
 &=& e^{\gamma_1\psi_{-}}|\nabla^{g_X}\psi_{-}|_{g_X}^2\leq C,
 \end{eqnarray*}
 which is made possible by assumption on $\psi_{-}$, the maximum principle applied to $$B\,e^{\gamma_0\psi_{-}}\dot{\f_t}^2+e^{\gamma_1\psi_{-}}|\nabla^{g(t)}\dot{\f_t}|_{g(t)}^2-Be^{Ct},$$ gives the desired result.
\end{proof}
Similar computations allows one to obtain a quantitative version of 
the local estimates on higher derivatives established in Proposition \ref{pro:higherbis}:
\begin{prop}\label{pro:alix2}
Assume $\psi_+=0$. If $|\partial\psi_{-}|_{\omega_X}\leq e^{-\alpha\psi_{-}}$ for some $\alpha> 0$,  then there exist constants $C$ and $\gamma>0$ such that for $t>0$,
\begin{equation*}
|\nabla^{g(t)}(\omega_t-\omega_X)|_{g(t)}\leq\frac{C}{\sqrt{t}}e^{-\gamma\psi_{-}+Ct}.
\end{equation*}
Moreover, for each $k\geq 0$, there exists $C_k>0$ and $\gamma_k>0$ such that for $t>0$,
\begin{equation*}
|\nabla^{g(t),\,k}\Rm(g(t))|_{g(t)}\leq\frac{C_k}{t^{1+\frac{k}{2}}}e^{-\gamma_k\psi_{-}+C_kt}.
\end{equation*}

\end{prop}

 \section{The metric structure induced by a current} \label{sec:metricspace}


 \subsection{The one-dimensional case} \label{sec:dim1}
 
 In this section, we consider a compact Kähler manifold $(X, \omega_X)=(S,\omega_S)$ of complex dimension $n=1$, i.e. $S$ is a compact Riemann surface. We denote by $d_S$ the geodesic distance on $S$ associated to $\om_S$. Next, we consider a positive current of the form 
 \begin{equation}
 \label{T dim 1}
 T=e^{\psi_+-\psi_-}\omega_S
 \end{equation}
 where $\psi_\pm$ are quasi-subharmonic functions on $S$. The functions $\psi_\pm$ are not uniquely defined, but one can and will always assume that for any $x\in S$, we have either $\nu(\psi_+,x)=0$ or $\nu(\psi_-,x)=0$. 
 
 For a quasi-subharmonic function $\psi$ on $S$, we consider the maximal value of Lelong numbers,
$$
\nu(\p):=\sup \{ \nu(\p,x), x \in S \}.
$$
Note that we have
\[ e^{-\p}\in L^1 \Longleftrightarrow e^{-\p} \in L^p  \,\, \mbox{for} \, \, \mbox{some } \, p>1\Longleftrightarrow \nu(\p)<2.\]
Let us observe further that the condition $e^{-\psi_-}\in L^1$ is stronger than the condition $e^{\psi_+-\psi_-}\in L^1$ even when $\nu(\p_+)=0$, as one can see by considering near $0\in \mathbb C$ the functions $\psi_+=-2\log(- \log |z|)$ and $\psi_-=2\log |z|$. It will be convenient to set
\[\nu_\pm:=\nu(\psi_\pm).\] 
  If $\gamma:[0,1]\to S$ is a Lipschitz path, we define
\[ \ell_{T}(\gamma):=\int_0^1   \|\dot{\gamma}(t)\|_{T} \,dt= \int_{0}^1 \|\dot \gamma(t)\|_{\omega_S} e^{\frac 12(\psi_+-\psi_-)(\gamma(t))}\,dt.\]
This is a well-defined, non-negative quantity which may however be infinite. 

\begin{defi} 
\label{def:distance1}
The semi-distance $d_{T}$ associated to $T$ is
$$
d_{T}(x,y)=\inf \left\{ \ell_{T}(\gamma), \; 
\gamma\subset S \text{ a Lipschitz path joining } x \text{ to } y \right\}.
$$
\end{defi} 

At this point, $d_T$ may take infinite values but the following theorem shows that it is essentially not the case.
Let us make the statement of Theorem \ref{thmD} more precise:

\begin{thm}
\label{dist dim 1}

If $\nu_-<2$ then $d_T$ is a well-defined distance  bi-Hölder equivalent to $d_S$.

\noindent
More precisely, fix $\e>0$ so small that $\alpha_{-}=\frac{1}{2}(\nu_-+\e)<1$ and set $ \alpha_{\pm}=\frac {1}{2}(\nu_{\pm}+\e)$.
Then there exists $C=C(\om_S, \p_{\pm}, \e)>0$ such that  for all paths $\gamma$, one has 
$$  
C^{-1} \ell_{\omega_S}^{1+\alpha_+}(\gamma) \leq \ell_{T}(\gamma) \leq C \ell_{\omega_S}^{1-\alpha_-}(\gamma).
$$
  In particular $C^{-1} d_{S}^{1+\alpha_+} \leq d_{T} \leq C d_{S}^{1-\alpha_-}$.
  \end{thm}

The proof of Theorem~\ref{dist dim 1} is given at the end of this section. It relies of the following result which establishes fine integrability properties of plane subharmonic functions.
 
 \begin{prop}  
 \label{prop:ExpInteg} 
Let $u$ be a  subharmonic function in an open set $\Omega \subset \C$ whose Lelong numbers 
satisfy $\nu^+:=\sup \{\nu (u,a) ; a \in \Omega \}<1$.
Fix a compact subset $K \subset \Omega$ , $\nu^+<\nu<1$, 
and $\gamma : [0,L]  \longrightarrow K$  a smooth curve parametrized  by arc length.  
Then
 $$
 \int_0^L e^{- u \circ \gamma(s)} d s \leq C(K,u)  L^{1- \nu}.
$$
  \end{prop}
  
 The constant $C(K,u)>0$ is moreover uniformly bounded if $u$ belongs to a compact subset of 
 subharmonic functions such that $\nu^+(u)<1$.

 \begin{proof} 
Let $ \mu_u := \frac{\Delta u }{2 \pi}  $ be the Riesz measure of $u$, and set $\nu_u(z,r) =  \mu_u(\bar{{\mathbb D }} (z,r))$.
Fix $R > 0$ such that $\Gamma:=\gamma([0,L]) \subset {\mathbb D } (0,R)$.
Since the Lelong numbers $\nu_u(z)=\inf_{r>0} \nu_u(z,r)$ are upper semi-continuous,
 we can assume that  $\nu (a,r) < \nu <1$ for $a \in K$ and $0< r <  r_1$, for $r_1 >0$ small.
  Replacing $u$ by  $u_{\e} (z)= u (z) + \e (\vert z\vert^2 - R^2)$
 if necessary, we can further assume that $ \nu_u(z,r)  \geq \delta > 0$ for all $z \in \Gamma$.

 The  Poisson-Jensen formula \cite[Proposition 1.25]{GZbook} shows  that for all 
 $(z,r) \in \Tilde \Omega:= \{(z,r) \in \Omega \times \R^+ ;  r < \text{dist}(z,\partial \Omega)\}$,
 $$
   -  u (z) =  -\int_0^{2 \pi} u (z+r e^{i \theta}) \frac{d \theta}{2 \pi}+
 \int_{{\mathbb D }(z,r)} \log (r\slash \vert z - \zeta\vert)  d\mu_u(\zeta).
$$
The monotonicity of circular mean values ensures that for $(z,r) \in \tilde \Omega$ 
 $$
    - \frac{1}{2 \pi} \int_0^{2 \pi} u (z+r e^{i \theta}) d \theta 
  \leq  - \frac{1}{\pi r^2} \int_{{\mathbb D }(z,r)}  u (\zeta) \, d \lambda_2 (\zeta)  
  \leq \frac{1}{\pi r^2} \Vert u\Vert_{L^1(\Omega)}.
$$
 
Setting $A (u) := \frac{1}{\pi r^2} \Vert u\Vert_{L^1(\Omega)}$
we thus obtain  
\begin{eqnarray*} 
 - u(z) &\leq & \frac{1}{\pi r^2} \Vert u\Vert_{L^1(\Omega)}   + \int_{{\mathbb D }(z,r)} \log (r\slash \vert z - \zeta\vert)  d\mu_u(\zeta) \\
 & \leq &  A(u) +  \int_{{\mathbb D }(z,r)} \nu_u(z,r) \log (r\slash \vert z - \zeta\vert)  \frac{d\mu_u(\zeta)}{\nu_u(z,r)}.
\end{eqnarray*}
It follows therefore from Jensen's inequality that
$$
\exp (- u (z)) \leq    \frac{1}{\delta}  e^{A (u)}\int_{{\mathbb D }(z,r)}  \frac{r^{\nu}}{ \vert z - \zeta\vert^\nu}   d \mu_u(\zeta).
$$
 Integrating the latter with respect to  $\lambda_\gamma=\gamma_*(ds)$
 and using Fubini's theorem, we obtain,
 setting $V_r(\Gamma) := \bigcup_{z\in \Gamma} {\mathbb D }(z,r)$,
  \begin{eqnarray*}
\int_\Gamma  \exp (- u (z) ) d \lambda_\gamma (z) 
&\leq & \delta^{-1}  e^{A(u)} \int_\Gamma   \int_{{\mathbb D }(z,r)}  \frac{r^{\nu}}{ \vert z - \zeta\vert^\nu} \mu_u(\zeta) d\lambda_\gamma(z) \\
& \leq & \delta^{-1}  e^{A(u)}
 \int_{V_r(\Gamma)} \int_\Gamma \frac{r^{\nu}}{ \vert z - \zeta\vert^\nu} d \lambda_\gamma (z) d \mu_u(\zeta)\\
 & = & \delta^{-1}  e^{A(u)} \int_{V_r(\Gamma)} J_{\Gamma}(\zeta) d \mu_u(\zeta),
 \end{eqnarray*}
where
 $ J_\Gamma(\zeta) := \int_\Gamma \frac{r^{\nu}}{ \vert z - \zeta\vert^\nu} d \lambda_\gamma (z)  $.
It follows from Lemma \ref{lem:density} below that
\begin{eqnarray*}
   J_\Gamma (\zeta)    
   &=&  \int_0^{+\infty} \lambda_\gamma  (\Gamma \cap {\mathbb D }(\zeta, r t^{-1 \slash \nu})) d t 
  \leq   \int_0^{+\infty} \min \{L, 8 r t^{-1 \slash \nu}\} d t \\
  &\leq & L \, t_1  +   (8 r) \int_{t_1}^{+\infty}  t^{- 1 \slash \nu} d t = B L^{1-\nu},
\end{eqnarray*}
 where $t_1= (\frac{8 r}{L})^\nu$ and  $B=\frac{(8 r)^\nu}{(1-\nu)}$.
Altogether this yields
 $$
  \int_0^L  \exp (- u \circ \gamma(s) ) d s \leq \frac{B}{\delta} L^{1-\nu} e^{\pi^{-1} r^{-2} \Vert u\Vert_{L^1(\Omega)}} \int_{V_r(\Gamma)} \Delta u. 
$$

Since $V_r(\Gamma) \subset  K' \Subset \Omega$ for $0< r $ small enough,
the desired inequality follows by observing that $\int_{K'} \Delta u \leq C_{K',\Omega} \Vert u\Vert_{L^1(\Omega)}$
(see \cite[Theorem 3.9]{GZbook}).
   \end{proof}

    
    \begin{lem} \label{lem:density} 
Fix $\zeta \in \Omega$ and $\rho >0$ such that ${\mathbb D }(\zeta,\rho) \subset \Omega$. Then  
  $$
  \lambda_\gamma (\Gamma \cap {\mathbb D }(\zeta,\rho)) \leq \min\{L,8 \rho\}\cdot
  $$
    \end{lem}
    
    \begin{proof}  
  Fix $\zeta \in \Omega$ and  $\rho >0$.  By definition,  
   $$
   \lambda_\gamma (\Gamma \cap {\mathbb D }(\zeta,\rho)) = \vert \{t \in [0,L] ; \gamma (t) \in {\mathbb D }(\zeta,\rho)\}\vert,
   $$
    where $\vert E \vert$ is the length  of a Borel set $E \subset \R$.
        We can assume   $\Gamma \cap {\mathbb D }(\zeta,\rho) \neq \emptyset$ and pick 
         $a  = \gamma (T) \in \Gamma \cap {\mathbb D }(\zeta,\rho)$ with $T \in [0,L]$.  Then $\Gamma \cap {\mathbb D }(\zeta,\rho) \subset \Gamma \cap {\mathbb D }(a,2\rho)$. Hence
  $$
  \{t \in [0,L] ; \gamma (t) \in {\mathbb D }(\zeta,\rho)\} \subset \{t \in [0,L] ; \gamma (t) \in {\mathbb D }(\gamma(T),2\rho)\}\cdot
  $$ 
  
    Fix an interval $I \subset [0,L]$ and assume that $\gamma$ is such that for all $s, t \in I$,
  \begin{equation} \label{eq:minorationgamma}
  \vert\gamma (t) - \gamma (s) \vert \geq (1\slash 2) \vert t - s\vert.
  \end{equation}  
 Set $\Gamma_I := \gamma(I) \subset \Gamma$.
Using the condition \eqref{eq:minorationgamma}, we obtain
  $$
  \{t \in I ; \gamma (t) \in {\mathbb D }(\zeta,\rho)\} \subset \{t \in I ; \gamma (t) \in {\mathbb D }(\gamma(T),2\rho)\} \subset I \cap {\mathbb D }(T,4\rho),
  $$ 
hence
   \begin{equation} \label{eq:densitygamma}
 \lambda_\gamma (\Gamma_I \cap {\mathbb D }(\zeta,\rho)) \leq  \vert I \cap {\mathbb D }(T,4\rho)\vert.
   \end{equation}
 
  We claim that there exists a finite number of intervals $(I_k)_{0 \leq k \leq N-1}$ with  mutually disjoint interiors such that  $\bigcup_{0 \leq k \leq N} I_k = [0,L]$ and for each $0 \leq k \leq N-1$, the restriction of $\gamma$ to  $I_k$ satisfies the condition \eqref{eq:minorationgamma}.
  Assuming this it follows from Property \eqref{eq:densitygamma} that 
  \begin{eqnarray*}
  \lambda_\Gamma (\Gamma \cap {\mathbb D }(\zeta,\rho)) 
  &\leq & \sum_{0\leq k \leq N-1} \lambda_\gamma (\Gamma_{I_k} \cap {\mathbb D }(\zeta,\rho)) 
   \leq  \sum_{0\leq k \leq N-1}  \vert I_k \cap {\mathbb D }(T,4\rho)\vert \\
  &=& \vert [0,L] \cap {\mathbb D }(T,4\rho)\vert \leq \min\{L, 8 \rho\},
  \end{eqnarray*}
   by additivity of the length.   \\

   It remains to prove the claim. Since $\gamma$ is parametrized by arclength,
   one has $\max \{\dot x (t)^2, \dot y (t)^2\} \geq  1/4$ for any $t \in [0,L]$.
One can thus find a finite sequence  $0 = T < t_1 < \cdots t_N = L$ such that  on each  
  $I_k := [t_k,t_{k+1}]$, either $ \dot x (t)^2  \geq 1\slash 4$ for any $t \in I_k$ 
  or $ \dot y (t)^2  \geq 1\slash 4$ for any $t \in I_k$.  
 The restriction of $\gamma$ to each   $I_k$ satisfies \eqref{eq:minorationgamma}.  
Indeed if $ \dot x (t)^2  \geq 1\slash 4$ on $I_k$,
then for all $t, s \in I_k$ there exists $c \in ]s,t[$ such that $x(t) - x(s) = (t-s) \dot x (c)$,
hence $\vert \gamma(t) - \gamma (s) \vert \geq \vert x(t) - x(s)\vert \geq  (1\slash 2)  \vert t-s \vert$.
 \end{proof}

    We are now ready for the proof of Theorem~\ref{dist dim 1}. 
    
    \begin{proof}[Proof of Theorem~\ref{dist dim 1}]
  Since $\p_+$ is bounded from above,  the upper bound 
  $\ell_{T}(\gamma) \leq C \ell_{\omega_S}^{1-\alpha^-}(\gamma)$ follows from Proposition~\ref{prop:ExpInteg} applied 
  to $u=\p^-/2$ in charts. 
  
  Similarly it suffices to treat the case $\p_-=0$ to establish the lower bound.
  Now, let $\gamma:[0,L] \mapsto S$ be a curve parametrized by arclength with respect to $\omega_S$. Without loss of generality, one can assume that the image of $\gamma$ is included in some coordinate chart $V\subset S$. 
Set $v=\p_+/2$; H\"older's inequality yields for  $p > 1$,
 \begin{eqnarray*}
  L &=&  \int_0^L e^{(1\slash p) v \circ \gamma(t)}  e^{- (1\slash p) v \circ \gamma(t)} d t \\
  &\leq & \left(\int_0^L e^{v \circ \gamma(t)} d t \right)^{1/p} \left(\int_0^L e^{- v \slash (p-1) \circ \gamma(t)} d t \right)^{1-1/p}.
 \end{eqnarray*}
 
Fix $p = 1 + \frac{\nu_+}{2(1-\delta)}$ where $\delta>0$ is very small, so that the function $u := \frac{\p_+}{p-1}$
satisfies $\sup \{\nu(u,x) ; x \in S \} =2(1-\delta)$.
Fix $0<\e<\delta$ and set  $\alpha=1-\delta+\e$.
Proposition~\ref{prop:ExpInteg}  yields
  $$
  \int_0^L e^{- v \slash (p-1) \circ \gamma(t)} d t \leq  C  e^{ C \Vert v \Vert_{L^1(V)}}L^{1 - \alpha}.
  $$
We infer  $c L^{1 + \alpha(p-1)} \leq  \ell_{T}(\gamma)$.
The conclusion follows since $\alpha(p-1)=\frac{\nu_+}{2}+\frac{\e \nu_+}{2(1-\delta)}$.
    \end{proof}

\subsection{Higher dimensional case}
\label{higher dim}

 In this section, we consider a compact Kähler manifold $(X, \omega_X)$ of complex dimension $n\ge2$. We denote by $d_X$ the geodesic distance on $X$ associated to $\om_X$. Next, we fix a proper analytic subset $Z\subsetneq X$ whose complement we denote by $\Omega$, a number $A>0$ and two functions $\psi_\pm\in \mathrm{PSH}(X,A\omega_X)$ such that the following holds: 
 \begin{enumerate}[label=$\bullet$]
 \item The functions $\psi_\pm$ are smooth in restriction to $\Omega$. 
 \item We have $\int_X e^{-\psi_-}\omega_X^n<+\infty$. 
\end{enumerate}
Note that the second condition can always be achieved up to scaling down $\psi_-$. 
Also, we can (and will) assume that $\int_X e^{\psi_+-\psi_-}\omega_X^n=\int_X \omega_X^n$. Moreover, it follows from the openness conjecture \cite{BoB15,GuanZhou15} that there exists $p>1$ such that 
\begin{equation}
\label{Lp}
\int_X e^{-p\psi_-}\omega_X^n<+\infty.
\end{equation} We consider the closed, positive current $T:=\omega_X+dd^c \varphi$ normalized by $\sup_X \varphi=0$ which is the unique solution of the Monge-Ampère equation 
\begin{equation}
\label{MA phi}
(\omega_X+dd^c \varphi)^n=e^{\psi_+-\psi_-}\omega_X^n
\end{equation}
provided by \cite{Kol98}. 
Note that $\varphi|_\Omega$ is smooth; in particular, $\omega:=T|_{\Omega}$ is a Kähler form on $\Omega$.   We define $d_{\omega}$ to be the geodesic distance associated with the possibly incomplete Kähler manifold $(\Omega, \omega)$.

If $\gamma:[0, 1]\to X$ is an arbitrary Lipschitz path in $X$, the integral $\int_0^1 \|\dot \gamma(t)\|_T\, dt$ need not be well-defined. The issue is that even though the coefficients of $T$ are functions in $L^p_{\rm loc}$ thanks to the estimate $T\le Ce^{-\psi_-}\omega_X$ from Theorem~\ref{thm:laplace}, their values along the image of $\gamma$ are not well-defined unlike in the case \eqref{T dim 1} studied above where the coefficients of $T$ were canonically defined at each point. However, if we assume additionally that $\gamma^{-1}(Z)$ is finite, then 
the quantity
\[\ell_T(\gamma):=\int_0^1 \|\dot \gamma(t)\|_T \,dt\]
is defined without ambiguity. This leads us to modifying Definition~\ref{def:distance1} as follows. 

\begin{defi} 
\label{def:distance2}
The semi-distance $d_{T}$ associated to $T$ is
$$
d_{T}(x,y)=\inf \left\{ \ell_{T}(\gamma), \; 
\gamma\, \text{Lipschitz path from } x \text{ to } y \,\, \mbox{s.t.} \, \,\gamma^{-1}(Z) \, \mbox{is finite} \right\}.
$$
\end{defi} 

When $\dim X=1$, the poles of $\psi_\pm$ are isolated hence Definition~\ref{def:distance2} is consistent with Definition~\ref{def:distance1}. The goals of the following proposition is to show that $d_T$ is finite, to determine sufficient conditions for $d_T$ to be a distance and to compare it with $d_\omega$. In particular, it encompasses Theorem \ref{thmE} in the introduction.

\begin{prop}
\label{prop dT}
In the setup above, we have
\begin{enumerate}[label=$(\roman*)$]
\item $d_{\omega}$ admits a unique continuous extension $\hat d_{\omega}$ to $X\times X$. Moreover, it satisfies 
$\hat d_{\omega} \le C d_X^\alpha$
for some constants $C>0$, $\alpha\in (0,1)$. 
\item $d_{T}$ coincides with $\hat d_{\omega}$, i.e.  $d_{T} = \hat d_{\omega}$. In particular, $d_{T}$ is finite.  
\item If $\psi_+$ has isolated singularities, $d_{T}$ is a distance. In particular, $(X, d_T)\simeq \overline{(\Omega, d_{\omega})}$, where the bar symbol denotes the metric completion. 
\item If $\psi_+$ has isolated analytic singularities, then $d_T$ is bi-Hölder equivalent to $d_X$.
\end{enumerate}
\end{prop}

\begin{proof}
$(i)$. First, we know that $\varphi$ is Hölder continuous by Theorem~\ref{thm:kolo}. Next, for any point $x_0\in \Omega$, the function $d_{\omega}(x_0, \cdot)$ extends to a function in $W^{1,2}(X,\omega_X)$ by \cite[Lemma~4.2]{GGZ25}. Therefore, one can run the proof of \cite[Proposition~1.4]{GGZ25} and see that
\begin{equation}
\label{hold control}
 d_{\omega} \le C d_X^\alpha \quad \mbox{on} \, \, \Omega \times \Omega,
\end{equation}
so that $(i)$ follows. 

$(ii)$. Let us first show that the inequality  
\begin{equation}
\label{hold T}
d_{T}\le Cd_X^{\alpha},
\end{equation}
holds on $X\times X$, maybe up to increasing $C$. Given $(i)$, the inequality $d_T|_{\Omega\times \Omega}\le d_\omega$ and the triangle inequality satisfied by $d_T$, is it enough to show \eqref{hold T} on $Z\times \Omega$. So we pick $(x,y)\in Z\times \Omega$ and set $\ep:=d_X(x,y)$. We can assume that $\ep$ is small enough so that $x,y$ belong to a coordinate chart. On that chart, one can replace $\omega_X$ with the euclidean metric without loss of generality. The segment joining $x$ and $y$ hits $Z$ finitely many times, hence one can always find a sequence of points $(y_k)_{k\ge 0}$ on that segment such that $y_0=y$ and $d_X(y_k,y_{k+1})\in[\frac{\ep}{2^{k+1}},\frac{\ep}{2^{k}}]$. Fix $\delta>0$ and let $\gamma_k=\gamma_k(\delta)$ be a smooth path in $\Omega$ joining $y_k$ to $y_{k+1}$ such that 
\[\ell_{\omega}(\gamma_k)\le d_{\omega}(y_k,y_{k+1})+\frac{\delta}{2^{k}}\le \frac{1}{2^k}(C\ep^\alpha+\delta)\]
where the second inequality follows from \eqref{hold control}. Concatenating all paths $\gamma_k$ yields a path $\gamma_\infty=\gamma_\infty(\delta)$ joining $y$ to $x$ whose interior lies in $\Omega$. Moreover, $\gamma_\infty$ is Lipschitz away from its endpoint and $\ell_{\omega}(\gamma_{\infty})\le 2(C\ep^\alpha+\delta)$. This implies that $d_{T}(x,y)\le 2Cd_X(x,y)^\alpha+2\delta$ for any $\delta>0$, hence $d_{T}\le 2Cd_X^\alpha$ on $Z\times \Omega$. Therefore \eqref{hold T} holds. 

Now, since the inequality $d_{T} \le d_{\omega}$ holds on $\Omega\times \Omega$ for obvious reasons, it extends automatically to $X\times X$ by continuity of $d_{T}$, the latter being guaranteed by \eqref{hold T}. 

Therefore, it remains to prove that $d_{T} \ge d_{\omega}$ say on $\Omega\times \Omega$.  Let $\ep>0$, let $x,y\in \Omega$ and let $\gamma$ be a path from $x$ to $y$ hitting $Z$ at most finitely many times such that $\ell_{T}(\gamma) \le d_{T}(x,y)+\ep$. In order to simplify the notation, let us assume that $\gamma$ hits $Z$ at a single point $p$ at a time $t_0$, and let $r>0$ small to be determined later. We can find $t_-<t<t_+$ such that $p_{\pm}:=\gamma(t_\pm)\in \partial B_p(r)$ where $B_p(r)$ is the ball or radius $r$ with respect to $\omega_X$. Let $\gamma_{\pm}$ be the restriction of $\gamma$ to $[t_-,t_+]$. We have $d_\omega(p_-, p_+)\le   C(2r)^\alpha$ from item $(i)$. In particular, there exists a path $\gamma_{\pm}^\circ$ connecting $p_-$ to $p_+$, lying entirely in $\Omega$ and such that $\ell_{\omega}(\gamma_{\pm}^\circ)\le  C(2r)^\alpha+\ep$. Replacing $\gamma_\pm$ by $\gamma_{\pm}^\circ$ yields a new path $\gamma^\circ$ from $x$ to $y$ lying entirely in $\Omega$ and such that 
  \[\ell_\omega(\gamma^\circ) \le d_{T}(x,y)+\ep+C(2r)^\alpha+\ep. \]
  Since $r$ can be chosen arbitrarily small, we get $d_\omega(x,y) \le d_{T}(x,y)$ as desired.

$(iii)$.  By compactness of $X$, the set $\Sigma$ of poles of $\p_+$ is finite. Now, let $x,y\in X$ be distinct points and let $r_0:=d_X(x,y)>0$. Assume for now that $x\notin \Sigma$ and let $0<r<\frac {r_0}2$ small enough so that the ball $B_x(r)$ of radius $r$ centered at $x$ with respect to $\omega_X$ is disjoint from $\Sigma$. Set $\kappa:=\inf_{B_{x}(r)} e^{\frac{\psi_+}{2}}>0$. By Theorem~\ref{thm:laplace}, we have $d_\omega(x,\cdot) \ge \kappa d_X(x,\cdot)$ on $B_x(r)$. By the choice or $r$, any path $\gamma$ from $x$ to $y$ hits $\partial B_x(r)$ hence $\ell_T(\gamma)\ge \kappa \cdot r>0$ showing the claim when $x\notin \Sigma$. If $x,y$ both belong to $\Sigma$, we pick $r <r_0/2$ small enough so that $B_x(r)\cap \Sigma=\{x\}$. By the choice of $r$, every path from $x$ to $y$ hits $\partial B_x(r)$. Set $A(x,r):=B_{x}(r)\setminus B_x(\frac r2)$, and define $\kappa:=\inf_{A(x,r)} e^{\frac{\psi_+}{2}}>0$. Again by Theorem~\ref{thm:laplace}, we have $d_\omega(S_x(\frac r2), S_x(r))\ge \kappa \cdot r >0$. This shows that $d_T$, hence $\hat d_{\omega}$ too, is a distance.  By item $(i)$, $\Omega$ is dense in $X$ for the topology induced by $\hat d_\omega$, therefore $(X,\hat d_\omega)$ is isometric to the metric completion of $(\Omega, d_\omega)$. 

$(iv)$. We can localize the analysis in a chart where such a pole is $0$. The result follows by observing that in that given chart, we have an inequality
$ T\ge C \|z\|^{2\alpha} dd^c \|z\|^2 $
   for some $\alpha >0$, thanks to Theorem~\ref{thm:laplace}. Now, the right-hand side is just the cone metric on $(\mathbb C^n,0)$ of radius $\|z\|^{1+\frac{\alpha}{2}}$, and the result follows. 
  \end{proof}

 When the singularities of $\p_+$ are spread along a divisor, obtaining a precise lower bound for $d_T$ becomes more subtle and should
 involve the assumption that the cohomology class of $T$ is K\"ahler;
 otherwise this divisor could be contracted as e.g. in \cite{SW13}. 
 We nevertheless expect that $d_{T}$ is always a distance, as the following example suggests.

\begin{exa} \label{exa:ramified}
Let $(Y,\omega_Y)$ be a compact Kähler manifold. Assume that there exists a line bundle $L$ on $Y$ such that for some $m\ge 2$, there exists a smooth divisor $B\in |mL|$ (e.g. if $L$ is semiample but not torsion). Let $\pi: X\to Y$ be the cyclic covering of degree $m$ branched along $B$, cf e.g. \cite[Proposition~4.1.6]{PAG1}. Then $X$ is smooth and $\pi^*B=mZ$ for some smooth divisor $Z$ on $X$. Consider the smooth semipositive $(1,1)$-form $T:=\pi^*\omega_Y$ on $X$. It is a Kähler form away from $Z$ but it is degenerate along $Z$. More precisely, if $(z_i)$ is a set of coordinates near a point $p\in Z$ such that $Z=(z_1=0)$, then $T$ is quasi-isometric to 
$|z_1|^{2(m-1)} idz_1\wedge d\bar z_1 + \sum_{j\ge 2} idz_j \wedge d\bar z_j$
near $p$. From this, it is easy to infer that the semi-distance $d_T$ from Definition~\ref{def:distance2} is actually a distance on $X$. 
\end{exa}

We finish this section by the following elementary lemma which will be useful later. 

\begin{lem} \label{ineq dist}
In the setup above, let $(\omega_j)_{j\ge 1}$ be a sequence of Kähler metrics such that $\omega_j|_{\Omega} \to \omega$ in $C^{\infty}_{\rm loc}(\Omega)$ as $j\to +\infty$. Then we have
$d_\omega \ge \limsup_{j\to +\infty} d_{\omega_j}$ on  $\Omega\times \Omega,$
where $d_{\omega_j}$ is the geodesic distance associated to $\omega_j$
\end{lem}

\begin{proof}
Let $x,y\in \Omega$, let $\ep>0$ and let $\gamma_\ep$ be a Lipschitz path  from $x$ to $y$ lying in $\Omega$ such that $\ell_\omega(\gamma_\ep)\le d_\omega(x,y)+\ep$. Let $U=U(\ep)\subset X$ be an open set containing the image of $\gamma_\ep$ and such that $U$ is relatively compact in $\Omega$. For $j\ge j_0(\ep)$, we have $\omega|_U \ge (1-\ep)^2 {\omega_j}|_U$, hence
\[d_{\omega}(x,y)\ge \ell_\omega(\gamma_\ep)-\ep \ge (1-\ep)\ell_{\omega_j}(\gamma_\ep)-\ep \ge  (1-\ep)d_{\omega_j}(x,y)-\ep,\]
hence the lemma follows.
\end{proof}


 \subsection{$(X,d_T)$ as Ricci limit space}\label{sec-ricci-limit}
 
 In this section, we aim to show that in the case where $\psi_+=0$, the metric space $(X,d_{T})$ naturally arises as Gromov-Hausdorff limit of compact Kähler manifolds with uniform Ricci lower bound. The main input is Cheeger-Colding theory, and more specifically the results of \cite{CCII}, \cite{LS18} and \cite{CJN}. 
  
 \subsubsection{Notations and setup}
 
We keep the notations from Section~\ref{higher dim} and assume additionally that $\psi_+=0$. In order to lighten notation, we set $\psi:=\psi_-$.  In this section, we make the convention that $C$ denotes a positive constant which may change from line to line but only depends on our backgroup setup $(X, \omega_X, T)$. 

Next, we consider a sequence $(\psi_j)_{j\ge 1}$ of smooth $A\omega_X$-psh functions decreasing pointwise to $\psi$ \cite{D92}, and we consider the Kähler form $\omega_j:=\omega_X+dd^c \varphi_j$ normalized by $\sup_X \varphi_j=0$ which is the unique solution of the Monge-Ampère equation 
\begin{equation}
\label{MA phi_j}
(\omega_X+dd^c \varphi_j)^n=e^{-\psi_j+c_j}\omega_X^n
\end{equation}
provided by \cite{Yau78}, where $c_j$ is such that $\int_X e^{-\psi_j+c_j}\omega_X^n=\int_X \omega_X^n$. Note that there exists $p>1$ such that we have 
\begin{equation}
\label{Lp2}
\int_X e^{-p(\psi_j+c_j)} \omega_X^n \le C
\end{equation} by \eqref{Lp}. Moreover, one can choose $\psi_j$ such that the convergence $\psi_j\to \psi$ is locally smooth on $\Omega$. In particular, we have local, smooth convergence $\omega_j|_{\Omega}\to \omega$ as tensors.  \\

We fix a constant $B>0$ such that $\Ric \omega_X \ge -B\omega_X$. It is well-known (cf Theorem~\ref{thm:laplace} at time zero) that there is a constant $C>0$ such that 
\begin{equation}
\label{lap esti}
C^{-1}\omega_X \le \omega_j \le Ce^{-\psi_j} \omega_X
\end{equation}
for all $j\ge 1$. In particular, it follows from \eqref{MA phi_j} that 
\begin{equation}
\label{ricci lb}
\Ric \omega_j \ge -(A+B)C\omega_j. 
\end{equation}

\subsubsection{Associated metric spaces}
Let us now discuss the various metric spaces associated with our above setup. We define $d_X$ (resp $d_j$) to be the geodesic distance associated with the Kähler form $\omega_X$ (resp. $\omega_j$) on $X$. It follows from \eqref{lap esti} and Theorem~\ref{thm:diameter} that we have a two-sided bound
\begin{equation}
\label{dist holder}
C^{-1}d_X \le d_j \le C d_X^{\alpha}
\end{equation}
for some $\alpha \in (0,1]$ and all $j\ge 1$. In particular, up to extracting a subsequence, we can assume that we have convergence 
\[(X,d_j)\underset{j\to +\infty}\longrightarrow (X_\infty, d_\infty)\]
 to a compact metric space $(X_\infty, d_\infty)$ in the Gromov-Hausdorff topology. It is straightforward to see from \eqref{dist holder} that $X_\infty$ is homeomorphic to $X$ and that $(X,d_X)$ and $(X_\infty, d_\infty)$ are bi-Hölder equivalent. In what follows, we will identify $X_\infty$ with $X$. \\

We aim to compare $(X, d_\infty)$ with the metric space $(X,d_T)$ associated with the current $d_T$ as considered in  Section~\ref{higher dim} above. Note that $d_T$ is indeed a distance thanks to Proposition~\ref{prop dT} and the latter is bi-Hölder equivalent to $d_X$; in particular $d_T$ induces the same topology as $d_X$ on $X$. Moreover, $d_T$ coincides with $d_\omega$ on $\Omega\times \Omega$, i.e. $(X,d_T)$ is the metric completion of $(\Omega, d_\omega)$.

\subsubsection{Properties of $(X,d_T)$}
The main goal of this section is to show the following theorem which implies Corollary \ref{coro-B}:

\begin{thm}
\label{thm metric completion}
In the above setup, we have $d_T=d_\infty$. In other words, the sequence $(X,d_j)$ converges to $(X,d_T)=\overline{(\Omega, d_\omega)} $ in the Gromov-Hausdorff topology. Moreover, $Z\subset (X,d_T)$ has Hausdorff dimension at most $2n-2$. 
\end{thm}

In another direction, let us mention that the general powerful result \cite[Theorem~5.1]{G+} implies the following
\begin{thm}[\cite{G+}]
The metric measure space $(X,d_T,T^n)$ is a non-collapsed $\mathrm{RCD}(2n,\lambda)$ space for some $\lambda \in \mathbb R$. 
\label{RCD}
\end{thm}

Indeed, $T$ can be approximated by the Kähler forms $\omega_j$ which satisfy the Ricci lower bound \eqref{ricci lb} hence the assumptions in {\it loc. cit.} are clearly met. Note that Theorem~\ref{RCD} above combined with the recent result \cite[Theorem~3]{Sz25} enable to recover the dimension bound $\dim_{\mathcal H}(Z)\le 2n-2$ in Theorem~\ref{thm metric completion}.

\begin{proof}[Proof of Theorem~\ref{thm metric completion}]

Recall from Lemma~\ref{ineq dist} that   $d_\omega \ge d_\infty$ on $\Omega\times \Omega$. By Proposition~\ref{prop dT}, this implies $d_T\ge d_\infty$ globally on $X\times X$. Therefore, the main task is to prove
\begin{equation}
\label{reverse ineq}
d_T \le d_{\infty}.
\end{equation}
We divide the proof in two steps where we crucially rely respectively on results of Cheeger-Colding \cite{CCII} and Cheeger-Jiang-Naber \cite{CJN}. 

\bigskip

{\bf Step 1. } {\it Reduction to a dimension estimate.}

\medskip

The sequence of compact Kähler manifolds $(X, \omega_j)$ satisfies
\begin{equation}
\label{geom bounds}
\Ric \omega_j \ge -C\omega_j, \quad \mathrm{vol}(X,\omega_j)=[\omega_X]^n, \quad \mathrm{diam}(X,\omega_j) \le C
\end{equation}
thanks to \eqref{ricci lb}, \eqref{MA phi_j} and \eqref{dist holder}. 
It follows from Bishop-Gromov comparison that 
\begin{equation}
\label{non collapse}
\mathrm{vol}(B_{\omega_j}(x,r))\ge C^{-1} r^{2n}
\end{equation}
for any $x\in X$ and any $0<r<1$. Let us introduce the metric regular set $\cR$ of the Gromov-Hausdorff limit $(X,d_\infty)$ (i.e. the set of points in $X$ all of whose tangent cones are isometric to the flat $\mathbb C^n$) and its complement, the singular set $\cS:=X\setminus\cR$. Note that $\cS$ may not be closed. The local smooth convergence $\omega_j|_{\Omega}\to \omega$ induces an injective, locally isometric embedding 
\[j:(\Omega, d_\omega)\longrightarrow (X,d_\infty)\]
and the main goal is to show that $j$ is actually isometric, from which the first part of Theorem~\ref{thm metric completion} follows since $j(\Omega)=\Omega$ is dense. From this local isometric embedding, we see that $\Omega\subset \cR$ but the inclusion may be 
strict.\footnote{
A typical local example in complex dimension one
is $T=(-\log |z|) idz\wedge d\bar z$. Indeed, its Ricci curvature has zero Lelong numbers hence the associated metric space is regular at $0$, cf \cite[Proposition~4.1]{LS18}.
} 
The following lemma shows that geodesic connectedness of $X\setminus Z$ is the key. 

\begin{lem}
\label{convexity}
Let $x\in \Omega$. Assume that there exists a dense set $\Omega'\subset \Omega$ such that for any $y\in \Omega'$, there exists a minimal geodesic for $d_\infty$ from $x$ to $y$ which lies in $\Omega$. Then $j$ is an isometric embedding.
\end{lem}

\begin{proof}[Proof of Lemma~\ref{convexity}]
Let $\gamma$ be such a minimal geodesic connecting $x$ and $y$, lying in $\Omega$. Since $j$ is locally isometric, we have
$d_\infty(x,y)=\ell_\omega(\gamma) \ge d_T(x,y)$
by definition of $d_T$. Therefore we have  $d_\infty\ge d_T$ on $\Omega'\times \Omega'$, hence \eqref{reverse ineq} holds everywhere by density of $\Omega' $ in $X$, and the result follows. 
\end{proof}

The next result shows that the connectedness property in Lemma~\ref{convexity} can be reduced to a dimension bound. 

\begin{lem}
\label{dimbdd}
Assume that $Z\subset (X,d_\infty)$ satisfies $\dim(Z\cap \cR) <2n-1$. Then the assumption in Lemma~\ref{convexity} is satisfied. 
\end{lem}

\begin{proof}[Proof of Lemma~\ref{dimbdd}]
Given the bounds \eqref{geom bounds} and \eqref{non collapse},  the results of \cite{CCI} show the upper bound 
\begin{equation}
\label{dim S} 
\dim \cS \le 2n-2.
\end{equation}
for the Hausdorff dimension of the singular set. Therefore, the assumption in the lemma implies that $\dim Z<2n-1$. 
The lemma now follows from \cite[Theorem~3.7]{CCII} applied to the closed subset $B=Z$ (note that in our non-collapsed situation, the measure $\nu_{-1}$ in {\it loc. cit.} is simply the $2n-1$-dimensional Hausdorff measure). 
\end{proof}

\bigskip

{\bf Step 2. } {\it The dimension estimate.}

\medskip

In this second step, we prove the estimate
\begin{equation}
\label{dim Z}
\dim(Z\cap \cR) \le 2n-2
\end{equation}
which will conclude the proof of the theorem by Lemma~\ref{dim S} in the previous step. Although $Z$ is a complex analytic subset hence has Hausdorff dimension at most $2n-2$ with respect to the euclidean metric, the statement is not obvious since on $\cR$ the metric $d_\infty$ may still be singular for lack of a Ricci upper bound on the $\omega_j$.

Let $p\in Z\cap \cR$ and let $\ep>0$ to be determined later. By definition of $\cR$, there exists $\delta=\delta(\ep)$ and 
$j_0=j_0(\ep)$ such that for $j\ge j_0$ we have $d_{\rm GH}(B_{\delta^{-2}\omega_j}(p,1), B_{\C^n}(0,1))\le \ep$. Thanks to \cite[Theorem~1.4]{LS18}, one can find such an $\ep=\ep(n)>0$ so that up to decreasing $\delta$ 
further there exists a holomorphic chart
\[F_j:B_{\delta^{-2}\omega_j}(p,1)\to \C^n\]
which is an $\ep$-GH approximation to its image. We will normalize $F_j$ by the condition $F_j(p)=0$. Appealing to (the proof of) \cite[Theorem~7.10]{CJN}, we see that $F_j$ satisfies the following distorsion estimate
\begin{equation}
\label{distorsion}
(1-\ep) \delta d_j(x,y)^{1+\ep} \le |F_j(x)-F_j(y)|\le (1+\ep)\delta d_j(x,y)
\end{equation}
for any $x,y\in B_{\delta^{-2}\omega_j}(p,1)$ and $j\ge j_0$, maybe up to decreasing $\ep,\delta$ again. Note that the image of $F_j$ contains (and is contained) in a ball of fixed radius centered at $0$. The above estimate allows us to find a limit map
$F_\infty:B_{\delta^{-1}d_\infty}(p,1)\to \C^n$
which satisfies 
for all $ \, x,y\in B_{\delta^{-1}d_\infty}(p,1)$, 
\begin{equation}
\label{distorsion 2}
(1-\ep) \delta d_\infty(x,y)^{1+\ep} \le |F_\infty(x)-F_\infty(y)|\le (1+\ep)\delta d_\infty(x,y),  
\end{equation}
In particular, $F_\infty$ induces a bi-Hölder homeomorphism onto its image. 

 Without loss of generality, one can assume that there exists a holomorphic function $\sigma$ on $B_{\omega_j}(p,\delta)$ such that $Z \cap B_{\omega_j}(p,\delta)\subset (\sigma=0)$ and $|\sigma|\le 1$, $|\sigma(q)|=\eta$ for some fixed $q\in B_{\omega_j}(p,\frac \delta 2)$ and $\eta>0$ . We set $\sigma_j:=(F_j)_*\sigma$. Then $\sigma_j$ subsequentially converges to a non-zero holomorphic function $\sigma_\infty$ on the image of $F_\infty$. If we set $Z_\infty := (\sigma_\infty=0) \subset \mathbb C^n$, then $F_\infty(Z)\subset Z_\infty$. In other words, $Z\cap B_{d_\infty}(p,\delta)$ is contained in the image of $Z_\infty$ by the map $F^{-1}_\infty$ which is Hölder continuous of exponent $\frac{1}{1+\ep}$. Since $Z_\infty$ is a complex hypersurface in $\mathbb C^n$, the estimate \eqref{distorsion 2} implies that 
$\mathcal H^{(2n-2)(1+\ep)}(Z \cap B_{d_\infty}(p,\delta))<+\infty.$
 Since $\ep$ can be taken arbitrarily small, the above is easily seen to imply \eqref{dim Z}, which concludes the proof of the theorem. 
\end{proof}

 \section{Gromov-Hausdorff convergence of the flow} \label{sec:GH}

 \subsection{Notation and upshot}
 \label{sec GH KRF}
 
As before, we let $(X, \omega_X)$ be a compact Kähler manifold, we let $T=\omega_X+dd^c \f$ be a closed positive $(1,1)$-current such that 
\begin{enumerate}[label=$\bullet$]
\item $\f$ is continuous and qpsh,
\item $T^n=e^{\p_+-\p_-} \omega_X^n$ for some qpsh functions $\p_{\pm}$,
\item $e^{-\p_-} \in L^1(\omega_X^n)$.
\end{enumerate}
For $t>0$ we consider the unique K\"ahler forms 
$\omega_t=\omega_X+dd^c \f_t$, where 
$$
(\omega_X+dd^c \f_t)^n=e^{\dot{\f_t}} \omega_X^n
\; \; \text{ and } \; \;
\f_t \stackrel{t \rightarrow 0}{\longrightarrow} \f.
$$
Let us denote by $d_t$ the geodesic distance induced by $\omega_t$ on $X$. It follows from Corollary \ref{cor:equicont1}  that the  compact metric spaces  
 $(X,d_{t})$ are relatively compact in the Gromov-Hausdorff topology, cf Proposition~\ref{pro:equicont} and the discussion below.
 We will study in this final section whether the latter  converge, as $t \rightarrow 0$,
 to a unique  compact metric space associated to $T$.

\subsection{Continuity of Ricci currents}

\subsubsection{Semicontinuity of the distances}

\begin{prop} \label{pro:equicont}
The functions $d_t: X \times X \rightarrow \R^+$ are equicontinuous.
\end{prop}

\begin{proof}
Corollary \ref{cor:equicont1} ensures that
$d_t \leq C d_X^{\alpha}$, where $C>0$ and $0 < \alpha \leq 1$.
For $(x,y), (x',y') \in X$, the triangle inequality yields
\begin{eqnarray*}
\left| d_t(x,y)-d_t(x',y') \right| &\leq &
\left| d_t(x,y)-d_t(x,y') \right|+\left| d_t(x,y')-d_t(x',y') \right| \\
& \leq & d_t(y,y')+d_t(x,x') \leq C d_X(x,x')^{\alpha}+C d_X(y,y')^{\alpha} \\
& \leq & C' \left[ d_X(x,x')+d_X(y,y') \right]^{\alpha}.
\end{eqnarray*}
\end{proof}

As $X \times X$ is compact and the diameters $(X,d_t)$ are bounded, it follows
from Arzela-Ascoli theorem that a subsequence 
$d_{t_i}$ uniformly converges on $X \times X$, as $t_i \rightarrow 0$, to a function
$d_0$ which is a semi-distance.
Thus $(X,d_t)$ converges, in the Gromov-Hausdorff sense, to
the quotient metric space $(X',d_0')$ induced by $d_0$
(see \cite[Example 7.4.4]{BBI01}).

We expect that $d_0$ coincides with the (semi-)distance $d_T$ studied in Section \ref{sec:metricspace}.
The following example shows that the weak convergence, as $t \rightarrow 0$, of $\omega_t$ towards $T$
is  too weak to ensure that $d_t$ converges to $d_{T}$ 
(see \cite{Top21} for an application to the Ricci flow).

\begin{exa}
Assume $(X,\omega_X)$ is a compact Riemann surface and fix, for each $j \in \N$, and finite set 
$\Lambda_j=\{a_1^{(j)},\ldots,a_{N_j}^{(j)} \} \subset X$ such that each point of $X$ lies at $d_{\omega_X}$ distance
at most $2^{-j}$ from $\Lambda_j$. Let $\gamma_{\ell k}^{(j)}$ be a minimizing geodesic for $\omega_X$ joining
$a_{\ell}^{(j)}$ to $a_k^{(j)}$ and set $\Sigma_j=\cup_{\ell,k} \gamma_{\ell k}^{(j)}$.
Construct a smooth function $\p_j$ on $X$ such that
\begin{itemize}
\item[$\bullet$] $-\ln 4=\p_j$ on $\Sigma_j$ and $-\ln 4 \leq \p_j \leq  \ln 4$ on $X$;
\item[$\bullet$] $\int_X e^{\p_j} \omega_X=\int_X \omega_X$ and $\int_X \left|e^{\p_j}-1 \right| \omega_X \leq 2^{-j}$.
\end{itemize}

By construction $\omega_j =e^{\p_j} \omega_X$ weakly converges to $\omega_X$ since 
$\p_j$ is uniformly bounded and converges to $0$ in $L^1(X)$, but we are going
to show that the distances $d_j$  associated to $\omega_j =e^{\p_j} \omega_X$ 
uniformly converge to $d/2$, where $d=d_{\omega_X}$.
Observe indeed that $\frac{1}{2} d \leq d_j \leq 2 d$
and for all $(x,y) \in \Sigma_j$,
$$
d_j(x,y) \leq \ell_{\omega_j}(\gamma_{x y}^{(j)})=\frac{1}{2} \ell_{\omega_X}(\gamma_{x y}^{(j)})=\frac{1}{2}d(x,y).
$$
Now if $(x,y) \in X$ one can find  $(x',y') \in \Sigma_j$ such that 
$d(x,x'), d(y,y') \leq 2^{-j}$ hence
\begin{eqnarray*}
\frac{1}{2} d(x,y) \leq d_j(x,y) & \leq &  d_j(x,x')+d_j(x',y')+d_j(y',y) \\
& \leq & 2 \,d(x,x')+\frac{1}{2} d(x',y')+ 2\,d(y',y) \\
& \leq & 4 \cdot2^{-j}+\frac{1}{2} \left[ d(x',x)+d(x,y)+d(y,y') \right] 
 \leq  5 \cdot 2^{-j}+\frac{1}{2} d(x,y).
\end{eqnarray*}
\end{exa}

Note that the functions $\p_j$ in the above example are far from being differences of $A\omega_X$-sh functions for any given $A>0$ independent of $j$.
We shall take advantage of this property in the next section and show that the
Ricci curvatures weakly converge to the signed current ${\rm Ric}(T)$.

\subsubsection{Ricci curvatures}

Recall that the Ricci curvature of $\omega_t$ is the differential form
$$
{\rm Ric}(\omega_t)={\rm Ric}(\omega_X)-dd^c \log \omega_t^n/\omega_X^n={\rm Ric}(\omega_X)-dd^c \dot{\f_t},
$$
while the Ricci curvature of $T$ is ${\rm Ric}(T)={\rm Ric}(\omega_X)-dd^c (\p_+-\p_-)$,
a closed bidegree $(1,1)$-current of order zero
(as a difference of positive closed currents). The following result establishes the first part of Theorem \ref{thmB}.

\begin{thm} \label{thm:cvricci}
The functions $\dot{\f}_t$ converge, as $t \rightarrow 0$, to $\p_+-\p_-$ in $L^1(\omega_X^n)$. In particular, the Ricci curvatures ${\rm Ric}(\omega_t)$ weakly converge to ${\rm Ric}(T)$ as $t \rightarrow 0$.
\end{thm}

\begin{proof}
We first show that $\dot{\f_t}$ weakly converges to 
$\p_+-\p_-$ as $t \rightarrow 0$.
Recall from Lemma \ref{lem:bddphidot} that there is $C>0$ such that for all $x \in X$ and $t>0$,
$$
\p_+(x)-C \leq \dot{\f_t}(x) \leq -\p_-(x)+C.
$$
In particular the families of functions $(e^{\dot{\f_t}})_{t>0}$ and
 $(\dot{\f_t})_{t>0}$ are weakly compact in $L^p$ for some $p>1$.
 
It follows from   \cite[Theorem 4.26]{GZbook}
that  $\omega_t^n =e^{\dot{\f_t}} \omega_X^n$ weakly converges towards $(\omega_X+dd^c \f)^n=e^{\p_+-\p_-} \omega_X^n$
as $t \rightarrow 0$, since $\f_t$ uniformly converges towards $\f$ by Theorem \ref{thm:GZ17}.
Thus $e^{\dot{\f_t}}$ weakly converges in $(L^p)^*$ to $e^{\p}$
as $t \rightarrow 0$, setting $\p=\p_+-\p_-$.
 
Let $u$ denote a weak-$(L^p)^*$ limit of some subsequence $(\dot{\f_{t_j}})$, i.e. $\dot{\f_{t_j}} \rightarrow u$,
and let $q$ denote the conjugate exponent of $p$.
We  fix $0 \leq \chi \in L^q$ normalized so that $\int_X \chi \omega_X^n=1$.
Jensen's inequality yields
$$
\int  e^{\p} \chi \omega_X^n =
\lim_j \int  e^{\dot{\f_{t_j}}} \chi \omega_X^n \geq \lim_j \exp  \left( \int_X  \dot{\f_{t_j}} \chi \omega_X^n \right)
=\exp  \left( \int_X  u \chi \omega_X^n \right).
$$
We apply this inequality to $\chi=\chi_{x,r}$ the (normalized) characteristic function of the ball $B_{\omega_X}(x,r)$.
By Lebesgue theorem $\int_X  u \chi_{x,r} \omega_X^n \stackrel{r \rightarrow 0}{\longrightarrow} u(x)$ for almost every $x$.

Letting $r$ decrease to zero we therefore obtain, for almost every $x \in X$, $e^{\p(x)} \geq e^{u(x)}$ hence $\p(x) \geq u(x)$.\\

It now suffices to show that
$\int_X \psi \omega_X^n \leq \int_X u \omega_X^n$
to conclude that there is equality $ u=\psi$.
We first observe that 
\begin{equation}
\label{non dc}
t \mapsto \int_X \dot{\f_t} \omega_X^n\quad \mbox{ is non-decreasing.}
\end{equation} Indeed
$$
e^{-\dot{\f_t}} dd^c \dot{\f_t}
=-dd^c e^{-\dot{\f_t}} + e^{-\dot{\f_t}} d\dot{\f_t} \wedge d^c \dot{\f_t} 
\geq -dd^c e^{-\dot{\f_t}},
$$
hence using $\omega_X^n=e^{-\dot{\f_t}} \omega_t^n$ and $\ddot{\f_t}=\Delta_t \dot{\f_t}$ we obtain
$$
\int_X  \ddot{\f_t} \omega_X^n =
\int_X e^{-\dot{\f_t}} dd^c \dot{\f_t} \wedge \omega_t^{n-1} 
\geq  \int_X -dd^c \left( e^{-\dot{\f_t}} \right) \wedge \omega_t^{n-1}=0,
$$
hence \eqref{non dc} holds. 
We show hereafter that 
\begin{equation}
\label{non dc 2}
\forall t>0, \quad  \int_X\psi \omega_X^n\le \int_X \dot{\f_t} \omega_X^n .
\end{equation}
The conclusion will then follow since $\int_X \dot{\f_{t_j}} \omega_X^n  \rightarrow \int_X u \omega_X^n$.\\

 We now use an approximation argument to show \eqref{non dc 2}.  Let $\p_j^{\pm}$ be smooth qpsh functions decreasing to $\p_{\pm}$, and set $\psi_j:=\psi_j^+-\psi_j^-$. 
 By \cite{Yau78} there exists a unique smooth $\omega$-psh function $\f_{0,j}$ such that
 $\sup_X \f_{0,j}=\sup_X \f$ and
 $$
 (\omega+dd^c \f_{0,j})^n=e^{\psi_j+c_j} \omega_X^n,
 $$
 where $c_j$ is a mass normalizing constant ($c_j \rightarrow 0$ as $j \rightarrow +\infty$).
 Since the densities $f_j=e^{\p_j +c_j}$ converge to $f=e^{\p}$ in $L^p$, 
 it follows from Kolodziej's stability estimate \cite{Kolstab} that 
 $\f_{0,j}$ uniformly converge to $\f$ as $j \rightarrow +\infty$.
 We now consider the smooth Monge-Amp\`ere parabolic potentials $\f_{t,j}$ such that
 $\f_{t,j} \rightarrow \f_{0,j}$ as $t \rightarrow 0$ and
 $$
 (\omega_X+dd^c \f_{t,j})^n =e^{\dot{\f}_{t,j}} \omega_X^n
 $$
 for $t>0$. Given \eqref{non dc}, we have
 $$\forall t>0, \quad 
   \int_X (\p_j+c_j) \omega_X^n
 =  \int_X \dot{\f}_{0,j} \omega_X^n \leq \int_X \dot{\f}_{t,j} \omega_X^n.
 $$
 By construction, $\int_X (\p_j +c_j) \omega_X^n \rightarrow \int_X \p \omega_X^n$
 as $j \rightarrow +\infty$,
 while it follows from \cite{GLZ20} that $\dot{\f}_{t,j}$ smoothly converge
 to $\dot{\f}_{t}$ for $t>0$. Thus $\int_X \dot{\f}_{t,j} \omega_X^n \rightarrow \int_X \dot{\f}_{t} \omega_X^n$
 and \eqref{non dc 2} follows, hence $u=\psi$ as desired.\\

 It remains to show that the convergence 
 $\dot{\f_t} \stackrel{t \rightarrow 0}{\longrightarrow} \p$ actually holds in $L^1$.
 The functions $e^{\dot{\f_t}/2}$ are weakly compact in $L^2$ as $t \rightarrow 0$.
 We let $f$ denote a cluster point. Reasoning as above, it follows from Jensen's inequality
 (applied to  $x \mapsto e^{x/2})$ that $f \geq e^{u/2}=e^{\p/2}$.
 On the other hand the same reasoning (Jensen's inequality applied to $x \mapsto x^2$) yields $e^{\p} \geq f^2$.
 Thus $f=e^{\p/2}$, i.e. $e^{\dot{\f_t}/2}$ weakly converges to $e^{\p/2}$.
 Since $\|e^{\dot{\f_t}/2}\|_{L^2} \rightarrow \|e^{\p/2}\|_{L^2}$ (as
 $e^{\dot{\f_t}}$ weakly converges to $e^{\p}$), we infer $$
 \|e^{\dot{\f_t}/2}-e^{\p/2}\|^2_{L^2}
 = \|e^{\dot{\f_t}/2}\|_{L^2}^2+\|e^{\p/2}\|_{L^2}^2-2 \langle e^{\dot{\f_t}/2}, e^{\p/2} \rangle \rightarrow 0.
 $$
 Extracting we can assume that $\dot{\f_t}(x)$ converges  to $\p(x)$ almost everywhere;
 the conclusion follows from Lebesgue dominated convergence theorem.
 \end{proof}

\begin{rem}
As the final step of the proof shows, the convergence of $\dot{\f}_t$ towards $\p_+-\p_-$ 
holds in $L^q$ for all $q>1$ (and even in the Orlicz space $L^{\rm exp}$).
\end{rem}

\subsection{The one dimensional case} \label{sec:alexandrov}

In this section we apply our analysis to the case of compact Riemann surfaces
and complete the proof of  Theorem \ref{thmD}.

\subsubsection{Alexandrov surfaces}\label{sec-alex-surf}


  We briefly review here the theory of compact surfaces with bounded integral curvature,
  as developped by the Leningrad school, following the survey article of
  Troyanov \cite{Troy09} and the book \cite{Re23} (see also \cite{CL25}).
  
  \begin{defi}
  Let $S$ be an oriented compact topological surface endowed with a geodesic distance $d$.
  We say that $(S,d)$ has bounded integral curvature if $d$ is the uniform limit
  of Riemannian geodesic distances $d_{g_i}$ such that for some uniform  $C>0$,
  $$
  \int_{S} |\kappa(g_i)| d A_i \leq C,
  $$
  where  $d A_i$ is the area measure of $(S,g_i)$,
  and $\kappa(g_i)$ denotes the Gauss curvature of $g_i$.
  \end{defi}
  
 The latter condition implies (by Banach-Alaoglu theorem) that the curvature
 measures $\kappa(g_i) dA_i$ weakly converge towards a unique
 (Radon) measure $\mu_d=\mu$ on $(S,d)$ which satisfies the Gauss-Bonnet theorem
 $$
 \mu(S)=\int_S d\mu =\lim_{i \rightarrow +\infty} \int_S \kappa(g_i) dA_i=\chi(S),
 $$
 where $\chi(S)$ denotes the Euler characteristic of $S$. 
 
 Conversely, one can show that
 any Radon measure $\mu$ on $S$ such that $\mu(S)=\chi(S)$ is
 the curvature metric of an Alexandrov metric $d$ on $S$.
  As $S$ is oriented, it admits a structure of compact Riemann surface.
    We equip  $S$ with a constant curvature metric $\omega_S$,
    i.e. such that $\Ric(\omega_S )= \kappa \omega_S$, where $\kappa \in \R$ is a constant.
    It satisfies $\int_S \kappa \,\omega_S=\chi(S)$. 
    The  Ricci curvature of   a positive Radon measure of the form $T_u=e^{u} \omega_S$  is defined to be
    $$
    \Ric(T_u)=\Ric(\omega_S)-dd^c \log (T_u/\omega_S)=\kappa \omega_S-dd^c u.
    $$
    Given  a Radon measure $\mu$
   such that $\mu(S)=\chi(S)$, it can be written
    $$
    \mu=\kappa \omega_S-dd^c u=\Ric(T_u),
    $$
    for a unique integrable function $u:S \rightarrow \overline{\R}$ such that
    $\int_S u \,\omega_S=0$. 
    
    Note that the function $u$ is a difference of 
    qsh functions.
      We let $\p_{\pm}$ denote two qsh functions on $S$
  (that is, $dd^c \p_{\pm} \geq -A \omega_S$ for some $A>0$) 
  such that $u=\p_+-\p_-$ and $\p_\pm$ have no common Lelong numbers.   
   \begin{defi}\label{defn-cusp}
 We say that a point $x$ is a cusp  if $\mu(x)=\Ric(T_u)(x) \geq 2$. 
 \end{defi}
 
We only consider the case when there is no cusp.
  This corresponds to the fact that 
  $e^{-\psi_-}\in L^p$ for some $p>1$.
 Consider
  $$
  d_{u}(x,y):=\inf \left\{ 
  \int_0^1 e^{\frac{1}{2} u \circ \gamma(t)} |\dot{\gamma}(t)|_{\omega_S} \,dt, \; \gamma(0)=x \text{ and } \gamma(1)=y 
  \right\}.
  $$
Proposition~\ref{prop:ExpInteg} ensures that 
   each integral $\int_0^1 e^{\frac{1}{2} u \circ \gamma(t)} |\dot{\gamma}(t)|_{\omega_S} \,dt$ is finite.
   The following is a recap on the main results obtained by Aleksandrov, Zalgaller and Reshetnyak 
(see \cite[Proposition 5.3, Theorem 6.1, Theorem 6.2, Corollary 6.3, Theorem 6.4]{Troy09}):

\begin{thm} \label{thm:recapdim1}
\text{ }

1) If there is no cusp then $d_u$ is a distance.

\smallskip

2) If the distances $d_j$  uniformly converge to $d$, then $\mu_{d_j}$ weakly converge to $\mu_d$.

\smallskip

3) If there is no cusp and $\mu_{d_j}$ weakly converge to $\mu_d$, then 
$d_{u_j}$ uniformly converges to $d_u$.

\smallskip

4) The distance determines the conformal structure.

\smallskip

5) If $(S,d)$ is an oriented compact Alexandrov surface with bounded integral curvature and without cusp, 
then there exists $\omega$ and $u$ as above
such that $d=d_u$.
\end{thm}

   \begin{defi}
  Let $S$ be an oriented compact topological surface endowed with a distance $d$.
  We say that $(S,d)$ has   curvature bounded from below if $d$ is the uniform limit
  of Riemannian distances $d_{g_i}$ such that
  $
  \kappa(g_i) \geq -C,
  $
  where $C>0$ is a uniform constant,  
  and $\kappa(g_i)$ denotes the Gauss curvature of $(S,g_i)$.
  \end{defi}
  
In this case  the curvature
 measures $\kappa(g_i) dA_i$ weakly converge towards a unique
  curvature metric $\mu_d=\mu$ of total mass $\chi(S)$, and such that
  $$
    \mu=\kappa \omega_S-dd^c u \geq -C \omega_S.
    $$
    This corresponds to the fact that $u=\p_+-\p_-$ with $\Delta_{\omega_S} \p_+ \in L^{\infty}$.
    Thus $-u$ is qsh and we can equivalently assume that $\p_+=0$.

\subsubsection{Canonical smoothing by the Ricci flow}

 The following result has been a source of inspiration for us:
 
 \begin{thm} \label{thm:richard}
 \cite{Sim12, Rich18}
 Let $S$ be a compact Alexandrov surface whose curvature is uniformly bounded below.
 Then $S$ is the initial datum of a smooth Ricci flow. Moreover the latter is unique
 up to diffeomorphism.
 \end{thm}
 
 We now extend this result to the general case of Alexandrov surface with bounded 
 integral curvature, relaxing the uniform bound from below on the Gauss curvature.

\begin{thm} \label{thm:fkr1dim}
Let $(S,d)$ be an oriented compact topological geodesic surface with  bounded integral curvature 
with no cusp.
Then $(S,d)$ is the initial datum of a smooth Ricci flow. 
\end{thm}

\begin{proof}
By assumption $S$ is a compact Riemann surface endowed with a metric $d_{T}$ associated
to a current $T=e^{\p_+-\p_-} \omega_S$, where 
\begin{itemize}
\item[$\bullet$] $\omega_S$ is a constant scalar curvature K\"ahler metric, ${\rm Ric}(\omega_S)=\kappa \omega_S$;
\item[$\bullet$] $\p_{\pm}$ are $A\omega_S$-psh functions for some $A>0$ large enough;
\item[$\bullet$]  $e^{-\p_-} \in L^p(\omega_S)$ for some $p>1$.
 \end{itemize}

To simplify the notations  we only treat the case of an elliptic curve. 
The K\"ahler-Ricci flow initiating from $T$ is then equivalent
to the complex Monge-Amp\`ere flow 
$$
\omega_t=\omega_S+dd^c \f_t=e^{\dot{\f_t}} \omega_S,
\; \; \text{ with } \; \; 
\f_t \rightarrow \f
\; \; \text{  at time zero}.
$$

Since there is no cusp, it follows from Theorem \ref{thm:recapdim1} that the distances $d_{t}$
uniformly converge to $d_{T}$ if and only if the Ricci curvatures weakly converge
$$
{\rm Ric}(\omega_t)=\kappa \omega_S-dd^c (\dot{\f_t}) 
\longrightarrow {\rm Ric}(T)=\kappa \omega_S-dd^c(\p_+-\p_-).
$$
The latter holds thanks to  Theorem \ref{thm:cvricci}.
\end{proof}

  \subsection{Geometric singularities}

The analysis is    more involved in higher dimension, so we 
further assume in this section that 
$\p_{\pm}$ are smooth in some Zariski open set, so 
that we can apply the results from Section~\ref{higher dim}.

It follows from Proposition \ref{prop dT} that $T$ induces a semi-distance $d_T$ on $X$. Moreover, if $\psi_+$ has isolated singularities, $d_T$ is actually a distance and it coincides with the metric completion of the smooth Riemannian metric defined by $T$ in $\Omega$.
The following result (in combination with Theorem \ref{thm:cvricci}) completes the proof of Theorem \ref{thmB} in the introduction.

\begin{thm} \label{thm:cvGHisolated}
Notation as in \textsection~\ref{sec GH KRF}. If $T$ has  geometric singularities and $\p_{\pm}$ have isolated singularities,
then 
the distances $d_{t}$ uniformly converge, as $t \rightarrow 0$, towards $d_{T}$.
In particular the compact metrics spaces $(X,d_{t})$ converge in the Gromov-Hausdorff sense
to  $(X,d_{T})$.
\end{thm}

\begin{proof}
Let $Z$ denote the (finite) set poles of $\p_{\pm}$.
It follows from Proposition \ref{pro:higherbis} that
$\omega_t$ converges locally smoothly to $T$ in $\Omega$. 
Also, we have seen in Proposition~\ref{prop dT} that $d_{T}$ is a distance and that 
\begin{equation}
\label{sm}
\mathrm{diam} (B_{\omega_X}(p,r), d_{T}) \underset{r\to 0} \longrightarrow 0,
 \quad \mathrm{for \,\, any} \, p\in X.
\end{equation}
Thanks to Lemma \ref{ineq dist}, it is enough to show that for any $x,y\in X$, we have 
\begin{equation}
\label{liminf}
\liminf_{t \rightarrow 0} d_t(x,y) \ge  d_{T}(x,y).
\end{equation}
We start by showing \eqref{liminf} whenever $x,y\in X\setminus Z$. We fix two such points $x,y$, as well as an arbitrarily small number $\ep>0$. Given \eqref{sm}, we can choose $r=r(\ep)$ such that 
\begin{equation}
\label{small balls}
\mathrm{diam} (B_{\omega_X}(p,r), d_{T}) \le \ep, \quad \mathrm{for \,\, any} \, p\in Z.
\end{equation}
Now, let $\gamma_t$ be a minimizing geodesic for $\omega_t$ joining $x$ to $y$; we can assume without loss of generality that $\gamma_t$ is parametrized by arclength with respect to $\omega_X$. 

We set $K:=\bigcup_{p\in Z} \bar B_{\omega_X}(p,r)$ and choose $t_{\ep}$ such that 
\begin{equation}
\label{minoration}
\forall \,0\le t \le t_{\ep}, \quad \omega_t \ge (1-\ep) T \quad \mathrm{on} \,\, X\setminus K,
\end{equation}
whose existence is guaranteed by smooth convergence of $\omega_t$ to $T$ in $X\setminus K$.

Let us first make a simple observation. If $\gamma_t$ hits some ball $B_{\omega_X}(p_0,r)$ for $p_0\in Z$ at time $s_0$ and another ball $B_{\omega_X}(p_1,r)$ ($p_1\in Z$, $p_1\neq p_0$) at time $s_1>s_0$, then $\gamma_t(s)$ does not hit $B(p_0,r)$ for any $s\ge s_1$, as long as $r$ is small enough. This is because the boundary of the two balls are at positive distance with respect to $d_t$ uniformly in $t$, so that \eqref{small balls} would allow to find a shorter path. 

Therefore, we can find a partition $[0,1]=\bigsqcup_{k=0}^{2N} I_{k}$ of intervals with $N\le |Z|$ such that
\[k\, \, \mbox{even}\quad  \Longrightarrow \quad  \forall s\in I_k, \gamma_t(s)\notin K.\] 
We write $\bar I_k=[s_k, s_{k+1}]$ with $0=s_0<s_1<\ldots <s_{2N+1}$, and we set $x_k:=\gamma_t(s_k)$. For any $1\le k \le N $, there exists $p_k\in Z$ such that $x_{2k-1}, x_{2k} \in \bar B_{\omega_X}(p_k,r)$ for some $p_k\in Z$. 

Now, consider the concatenation of paths $\tilde \gamma_t:=\cup_{k=0}^N \gamma_t(I_{2k})$ whose image is disjoint from $K$. We have
\begin{eqnarray*}
d_t(x,y) & \ge & \ell_{\omega_t}(\tilde \gamma_t) 
\ge (1-\ep) \ell_{T}(\tilde \gamma_t)\\
&\ge & (1-\ep) \sum_{k=0}^N d_{T}(x_{2k}, x_{2k+1}) \\
& \ge & (1-\ep) \big(d_{T}(x,y)-\sum_{k=1}^N d_{T}(x_{2k-1}, x_{2k})\big)\\
&\ge & (1-\ep)d_{T}(x,y)-(1-\ep)N\ep.
\end{eqnarray*}
where we have used successively \eqref{minoration}, the triangle inequality and \eqref{small balls}. 
As $\ep$ was arbitrary and $N\le |Z|$, we get \eqref{liminf} for any $x,y\notin Z$. 

Fix  $x \in X \setminus Z$ and $y \in Z$. We check that 
$d_t(x,y) \rightarrow d_{T}(x,y)$ as $t \rightarrow 0$. 
Fix $\e>0$ and some large constant $C$ to be determined later. By \eqref{small balls}  one can pick $z \in X \setminus Z$ such that $ Cd_X(y,z)^{\alpha}+d_{T}(y,z) <\e$.
Thus
\begin{eqnarray*}
\left| d_{T}(x,y)-d_t(x,y) \right|
& \leq & \left| d_{T}(x,y)-d_{T}(x,z) \right|+\left| d_{T}(x,z)-d_t(x,z) \right|+\left| d_{t}(x,z)-d_t(x,y) \right| \\
& \leq &  d_{T}(y,z) +\left| d_{T}(x,z)-d_t(x,z) \right|+d_{t}(y,z) \\
& \leq &  d_{T}(y,z) +\left| d_{T}(x,z)-d_t(x,z) \right|+C d_{X}(y,z)^{\alpha}  \\
& \leq & \e+\left| d_{T}(x,z)-d_t(x,z) \right|,
\end{eqnarray*}
where we have used Theorem \ref{thm:diameter}. The convergence follows from the previous step.

When $x,y \in Z$ we approximate $x$ by $x' \in X \setminus Z$ and use the same argument and the previous step to show that
$d_t(x,y) \rightarrow d_{T}(x,y)$. The proof is complete.
\end{proof}

\bibliographystyle{smfalpha}
\bibliography{biblioFKR}

@article {Naber-Open-Pb,
    AUTHOR = {Naber, Aaron},
     TITLE = {Conjectures and open questions on the structure and regularity
              of spaces with lower {R}icci curvature bounds},
   JOURNAL = {SIGMA},
    VOLUME = {16},
      YEAR = {2020},
     PAGES = {Paper No. 104, 8},
      ISSN = {1815-0659},
   MRCLASS = {53C21 (53C23)},
  MRNUMBER = {4164873},
       DOI = {10.3842/SIGMA.2020.104},
       URL = {https://doi.org/10.3842/SIGMA.2020.104},
}

@unpublished{ Troy09,
	author = "M. Troyanov",
	title = {{Les surfaces \`a courbure int\'egrale born\'ee au sens d'Alexandrov}},
	year = "2009",
	Note = {Preprint \href{http://arxiv.org/abs/0906.3407}{arXiv:0906.3407}}
}

@article{FGS20,
	author = "Fu, Xin and Guo, Bin and Song, Jian",
	title = {{Geometric estimates for complex Monge-Ampère equations}},
	JOURNAL = {J. Reine Angew. Math.},
	 FJOURNAL = {Journal f\"{u}r die Reine und Angewandte Mathematik. [Crelle's
              Journal]},
      VOLUME = {765},
      YEAR = {2020},
     PAGES = {69–99}
}

@article{CL25,
	author = "Chen, Jingyi and Li, Yuxiang",
	title = {Uniform convergence of metrics on {A}lexandrov surfaces with bounded integral curvature},
	JOURNAL = {Adv. Math.},
      VOLUME = {479},
      YEAR = {2025},
     PAGES = {110436}
}

@article{TY24,
  Author =	 {Topping, P. and Yin, H.},
  Title =	 {{Uniqueness of Ricci flows from non-atomic Radon measures on Riemann surfaces}},
  JOURNAL = {Proc. Lond. Math. Soc. },
  YEAR = {2024},
    NUMBER = {128},
     PAGES = {30 pp}
}

@incollection {Top21,
    AUTHOR = {Topping, Peter},
     TITLE = {{Loss of initial data under limits of Ricci flows}},
 BOOKTITLE = {Minimal surfaces: integrable systems and visualisation},
    VOLUME = {349},
     PAGES = {257–261},
 PUBLISHER = {Springer},
      YEAR = {2021},
}

@article{GPSS24,
	author = "Guo,B. and Phong,D.H. and Song,J. and Sturm,J.",
	title = { {Diameter estimates in K\"ahler geometry}},
	JOURNAL = {Comm. Pure Appl. Math.},
	VOLUME = {77},
	NUMBER = {8},
	PAGES = {3520–3556}, 
	year = "2024"	
}

@article {GZ17,
    AUTHOR = {Guedj, Vincent and Zeriahi, Ahmed},
     TITLE = {Regularizing properties of the twisted {K}\"{a}hler-{R}icci flow},
   JOURNAL = {J. Reine Angew. Math.},
  FJOURNAL = {Journal f\"{u}r die Reine und Angewandte Mathematik. [Crelle's
              Journal]},
    VOLUME = {729},
      YEAR = {2017},
     PAGES = {275--304},
      ISSN = {0075-4102},
   MRCLASS = {53C44 (53C55)},
  MRNUMBER = {3680377},
MRREVIEWER = {Yuanqi Wang},
       DOI = {10.1515/crelle-2014-0105},
       URL = {https://doi.org/10.1515/crelle-2014-0105},
}

@book {GZbook,
    AUTHOR = {Guedj, Vincent and Zeriahi, Ahmed},
     TITLE = {Degenerate complex {M}onge-{A}mp\`ere equations},
    SERIES = {EMS Tracts in Mathematics},
    VOLUME = {26},
 PUBLISHER = {European Mathematical Society (EMS), Z\"{u}rich},
      YEAR = {2017},
     PAGES = {xxiv+472},
      ISBN = {978-3-03719-167-5},
   MRCLASS = {32W20 (32Q20 32U15 32U20 32U40 35J96)},
MRREVIEWER = {Slimane Benelkourchi},
       DOI = {10.4171/167},
       URL = {https://doi.org/10.4171/167},
}

@book {BBI01,
    AUTHOR = {Burago, Dmitri and Burago, Yuri and Ivanov, Sergei},
     TITLE = {A course in metric geometry},
    SERIES = {Graduate Studies in Mathematics},
    VOLUME = {33},
 PUBLISHER = {American Mathematical Society, Providence},
      YEAR = {2001},
     PAGES = {xiv+415},
      ISBN = {0-8218-2129-6},
}

@book {Re23,
    AUTHOR = {François Fillastre and Dmitriy Slutskiy},
     TITLE = {Reshetnyak's Theory of Subharmonic Metrics},
    SERIES = {CRM Springer},
 PUBLISHER = {Springer},
      YEAR = {2023},
      }

@unpublished{ G+,
	author = "Chen, Yifan and Chiu, Shih-Kai and Hallgren, Max and Székelyhidi, Gábor and Tô, Tat Dat and Tong, Freid",
	title = {{On Kähler-Einstein Currents}},
	year = "2025",
	Note = {Preprint \href{http://arxiv.org/abs/2502.09825}{arXiv:2502.09825}}
}

@article{GGZ25,
	author = "Vincent Guedj and Henri Guenancia and Ahmed Zeriahi",
	title = {{Diameter of K\"ahler currents}},
	JOURNAL = {J. Reine Angew. Math.},
	VOLUME = {80},
	PAGES = {115-152}, 
	year = "2025"	
}

@article {GLZ20,
    AUTHOR = {Guedj, Vincent and Lu, Hoang Chinh and Zeriahi, Ahmed},
     TITLE = {Pluripotential Kähler-Ricci flows},
   JOURNAL = {Geom. Topol.},
  FJOURNAL = {Geometry and Topology},
    VOLUME = {24},
      YEAR = {2020},
     PAGES = {1225-1296},
}

@article {GuanZhou15,
    AUTHOR = {Guan, Qi'an and Zhou, Xiangyu},
     TITLE = {A proof of {D}emailly's strong openness conjecture},
   JOURNAL = {Ann. of Math. (2)},
  FJOURNAL = {Annals of Mathematics. Second Series},
    VOLUME = {182},
      YEAR = {2015},
    NUMBER = {2},
     PAGES = {605--616},
      ISSN = {0003-486X},
   MRCLASS = {32U05},
  MRNUMBER = {3418526},
MRREVIEWER = {\.Zywomir Dinew},
       DOI = {10.4007/annals.2015.182.2.5},
       URL = {http://dx.doi.org/10.4007/annals.2015.182.2.5},
}

@incollection {BoB15,
    AUTHOR = {Berndtsson, Bo},
     TITLE = {The openness conjecture and complex {B}runn-{M}inkowski
              inequalities},
 BOOKTITLE = {Complex geom. and dyn.},
    SERIES = {Abel Symp.},
    VOLUME = {10},
     PAGES = {29--44},
 PUBLISHER = {Springer},
      YEAR = {2015},
   MRCLASS = {32U05 (28A75 31C10 52A40)},
  MRNUMBER = {3587460},
MRREVIEWER = {Steven George Krantz},
}

@article {ChCo97,
    AUTHOR = {Cheeger, Jeff and Colding, Tobias H.},
     TITLE = {On the structure of spaces with {R}icci curvature bounded
              below. {I}},
   JOURNAL = {J. Differential Geom.},
  FJOURNAL = {Journal of Differential Geometry},
    VOLUME = {46},
      YEAR = {1997},
    NUMBER = {3},
     PAGES = {406--480},
      ISSN = {0022-040X},
   MRCLASS = {53C21 (53C20)},
  MRNUMBER = {1484888},
MRREVIEWER = {William P. Minicozzi, II},
       URL = {http://projecteuclid.org/euclid.jdg/1214459974},
}

@Article{ D92,
	AUTHOR = "Jean-Pierre Demailly",
	TITLE = "Regularization of closed positive currents and intersection theory",
	JOURNAL = "J. Algebraic Geom.",
	FJOURNAL = "Journal of Algebraic Geometry",
	VOLUME = "1",
	YEAR = "1992",
	NUMBER = "3",
	PAGES = "361--409",
	ISSN = "1056-3911",
	MRCLASS = "32C30 (32C17 32J25)",
	MRNUMBER = "1158622 (93e:32015)",
	MRREVIEWER = "Takeo Ohsawa"
}

@Article{ Pau08,
	author = "Mihai P{\u{a}}un",
	title = {{Regularity properties of the degenerate Monge-Amp{\`e}re equations on compact K{\"a}hler manifolds.}},
	journal = "Chin. Ann. Math., Ser. B",
	year = 2008,
	volume = "29",
	number = 6,
	pages = "623--630"
}

@article {Rich18,
    AUTHOR = {Richard, Thomas},
     TITLE = {Canonical smoothing of compact {A}leksandrov surfaces via
              {R}icci flow},
   JOURNAL = {Ann. Sci. \'{E}c. Norm. Sup\'{e}r. (4)},
  FJOURNAL = {Annales Scientifiques de l'\'{E}cole Normale Sup\'{e}rieure. Quatri\`eme
              S\'{e}rie},
    VOLUME = {51},
      YEAR = {2018},
    NUMBER = {2},
     PAGES = {263--279},
      ISSN = {0012-9593},
   MRCLASS = {53C44 (53C45)},
  MRNUMBER = {3798303},
MRREVIEWER = {Yong Huang},
       DOI = {10.24033/asens.2356},
       URL = {https://doi.org/10.24033/asens.2356},
}

@article {Sim12,
    AUTHOR = {Simon, Miles},
     TITLE = {Ricci flow of non-collapsed three manifolds whose {R}icci
              curvature is bounded from below},
   JOURNAL = {J. Reine Angew. Math.},
  FJOURNAL = {Journal f\"{u}r die Reine und Angewandte Mathematik. [Crelle's
              Journal]},
    VOLUME = {662},
      YEAR = {2012},
     PAGES = {59--94},
      ISSN = {0075-4102},
   MRCLASS = {53C44},
  MRNUMBER = {2876261},
MRREVIEWER = {Takumi Yokota},
       DOI = {10.1515/CRELLE.2011.088},
       URL = {https://doi.org/10.1515/CRELLE.2011.088},
}

@unpublished{GPTW21,
  Author =	 {Guo, Bin and Phong, Duong Hong and Tong, Freid and Wang, Chuwen},
  Note =	 {Preprint  \href{https://arxiv.org/abs/2112.02354}{arXiv:2112.02354}},
  Title =	 {{ On the modulus of continuity of solutions to complex Monge-Amp\`ere equations}},
  Year =	 2021
}

@article{ST17,
  Author =	 {M. Simon and Topping, P.},
  Title =	 {{ Local mollification of Riemannian metrics using Ricci flow, and Ricci limit spaces}},
  JOURNAL = {Geom. Topol.},
  YEAR = {2021},
    NUMBER = {25},
     PAGES = {913–948},
}

@article{CJN,
 author = {Cheeger, Jeff and Jiang, Wenshuai and Naber, Aaron},
 title = {Rectifiability of singular sets of noncollapsed limit spaces with {Ricci} curvature bounded below},
 fjournal = {Annals of Mathematics. Second Series},
 journal = {Ann. Math. (2)},
 issn = {0003-486X},
 volume = {193},
 number = {2},
 pages = {407--538},
 year = {2021},
 zbMATH = {7331714},
 Zbl = {1469.53083}
}

@article{CCI,
 author = {Cheeger, Jeff and Colding, Tobias H.},
 title = {On the structure of spaces with {Ricci} curvature bounded below. {I}},
 fjournal = {Journal of Differential Geometry},
 journal = {J. Differ. Geom.},
 issn = {0022-040X},
 volume = {46},
 number = {3},
 pages = {406--480},
 year = {1997},
 doi = {10.4310/jdg/1214459974},
 zbMATH = {1145657},
 Zbl = {0902.53034}
}

@article{CCII,
 author = {Cheeger, Jeff and Colding, Tobias H.},
 title = {On the structure of spaces with {Ricci} curvature bounded below. {II}},
 fjournal = {Journal of Differential Geometry},
 journal = {J. Differ. Geom.},
 issn = {0022-040X},
 volume = {54},
 number = {1},
 pages = {13--35},
 year = {2000},
 doi = {10.4310/jdg/1214342145},
 zbMATH = {1782634},
 Zbl = {1027.53042}
}

@article {ST09,
    AUTHOR = {Song, Jian and Tian, Gang},
     TITLE = {The {K}\"{a}hler-{R}icci flow through singularities},
   JOURNAL = {Invent. Math.},
  FJOURNAL = {Inventiones Mathematicae},
    VOLUME = {207},
      YEAR = {2017},
    NUMBER = {2},
     PAGES = {519--595},
      ISSN = {0020-9910},
   MRCLASS = {32Q15 (14E30 32W20 53C44 53C55)},
  MRNUMBER = {3595934},
MRREVIEWER = {Haozhao Li},
       DOI = {10.1007/s00222-016-0674-4},
       URL = {https://doi.org/10.1007/s00222-016-0674-4},
}

@Article{ Kol98,
	author = "S{\l}awomir Ko{\l}odziej",
	title = "{The complex Monge-Amp{\`e}re operator}",
	journal = "Acta Math.",
	year = "1998",
	volume = "180",
	number = "1",
	pages = "69--117"
}

@Article{ Kol08,
	title = {{H{\"o}lder continuity of solutions to the complex Monge-Amp{\`e}re equation with the right-hand side in $L^p$: the case of compact K{\"a}hler manifolds}},
	author = "S{\l}awomir Ko{\l}odziej",
	journal = "Math. Ann.",
	pages = "379--386",
	volume = "342",
	number = "1",
	year = "2008"
}

@article {Kolstab,
    AUTHOR = {S{\l}awomir Ko{\l}odziej},
     TITLE = {The {M}onge-{A}mp\`ere equation on compact {K\"a}hler manifolds},
   JOURNAL = {Indiana Univ. Math. J.},
  FJOURNAL = {Indiana University Mathematical Journal},
    VOLUME = {52},
      YEAR = {2003},
    NUMBER = {3},
     PAGES = {667–686},
}

@Article{ EGZ09,
	author = "Philippe Eyssidieux and Vincent Guedj and Ahmed Zeriahi",
	title = {{Singular K{\"a}hler-Einstein metrics}},
	journal = "{J. Amer. Math. Soc.}",
	year = "2009",
	pages = "607--639",
	volume = "22"
}

@article {SW13,
    AUTHOR = {Jian Song and Ben Weinkove},
     TITLE = {Contracting exceptional divisors by the {K\"a}hler-{R}icci flow},
   JOURNAL = {Duke Math. J.},
  FJOURNAL = {Duke Mathematical Journal},
    VOLUME = {162},
      YEAR = {2013},
    NUMBER = {2},
     PAGES = {367-415},
}

@unpublished{ Sz25,
	author = "G{\'a}bor Sz{\'e}kelyhidi",
	title = { { Gromov-Hausdorff limits of collapsing Calabi-Yau fibrations }},
	Note = {Preprint \href{http://arxiv.org/abs/2505.14939}{arXiv:2505.14939}},
	year = "2025"
}

@article{ Li21,
	author = "Yang Li",
	title = { {On collapsing Calabi-Yau fibrations}},
	JOURNAL = {J. Differential Geom.},
    VOLUME = {117},
      YEAR = {2021},
    NUMBER = {3},
     PAGES = {451-483},
}

@article{ LS18,
	author = "Liu,G. and Sz{\'e}kelyhidi,G.",
	title = { {Gromov-Hausdorff limits of K\"ahler manifolds with Ricci curvature bounded below}},
	 JOURNAL = {Geom. Funct. Anal.},
    VOLUME = {32},
      YEAR = {2022},
    NUMBER = {2},
     PAGES = {236–279},
}

@inproceedings {Chern,
    AUTHOR = {Chern, Shiing-shen},
     TITLE = {On holomorphic mappings of hermitian manifolds of the same
              dimension },
 BOOKTITLE = {Entire {F}unctions and {R}elated {P}arts of {A}nalysis
              ({P}roc. {S}ympos. {P}ure {M}ath., {L}a {J}olla, {C}alif.,
              1966)},
     PAGES = {157--170},
 PUBLISHER = {Amer. Math. Soc., Providence, R.I.},
      YEAR = {1968},
   MRCLASS = {53.80 (32.00)},
  MRNUMBER = {0234397},
MRREVIEWER = {W. Barth},
}

@article {DDGHKZ14,
    AUTHOR = {Demailly, Jean-Pierre and Dinew, S{\l}awomir and Guedj,
              Vincent and Pham, Hoang Hiep and Ko{\l}odziej, S{\l}awomir and
              Zeriahi, Ahmed},
     TITLE = {H\"older continuous solutions to {M}onge-{A}mp\`ere equations},
   JOURNAL = {J. Eur. Math. Soc. (JEMS)},
  FJOURNAL = {Journal of the European Mathematical Society (JEMS)},
    VOLUME = {16},
      YEAR = {2014},
    NUMBER = {4},
     PAGES = {619--647},
      ISSN = {1435-9855},
   MRCLASS = {32W20 (32Q15 32U05 32U15 32U40 35B65 35J96 53C55)},
  MRNUMBER = {3191972},
MRREVIEWER = {Muhammed Ali Alan},
       DOI = {10.4171/JEMS/442},
       URL = {http://dx.doi.org/10.4171/JEMS/442},
}

@Article{ Yau78,
	author = "Shing-Tung Yau",
	title = {{On the Ricci curvature of a compact K{\"a}hler manifold and the complex Monge-Amp{\`e}re equation. I.}},
	journal = "Commun. Pure Appl. Math.",
	volume = "31",
	pages = "339--411",
	year = "1978"
}

@Article{ GSS,
	author = "Henri Guenancia",
	title = "{Semistability of the tangent sheaf of singular varieties}",
	journal = "Algebraic Geometry",
	year = "2016",
	month = "november",
	pages = "508-542",
	number = "5",
	volume = "3"
}

@Book{ Siu87,
	author = "Yum-Tong Siu",
	title = {{Lectures on Hermitian-Einstein Metrics for Stable Bundles and K{\"a}hler-Einstein Metrics}},
	publisher = {Birkh{\"a}user},
	year = "1987"
}

@book {PAG1,
    AUTHOR = {Lazarsfeld, Robert},
     TITLE = {Positivity in algebraic geometry. {I}},
    SERIES = {Ergebnisse der Mathematik und ihrer Grenzgebiete. 3. Folge. A
              Series of Modern Surveys in Mathematics},
    VOLUME = {48},
      NOTE = {Classical setting: line bundles and linear series},
 PUBLISHER = {Springer-Verlag, Berlin},
      YEAR = {2004},
     PAGES = {xviii+387},
      ISBN = {3-540-22533-1},
   MRCLASS = {14-02 (14C20)},
  MRNUMBER = {2095471},
MRREVIEWER = {Mihnea Popa},
       DOI = {10.1007/978-3-642-18808-4},
       URL = {http://dx.doi.org/10.1007/978-3-642-18808-4},
}

@article {DnL17,
    AUTHOR = {Di Nezza, Eleonora and Lu, Chinh H.},
     TITLE = {Uniqueness and short time regularity of the weak
              {K}\"{a}hler-{R}icci flow},
   JOURNAL = {Adv. Math.},
  FJOURNAL = {Advances in Mathematics},
    VOLUME = {305},
      YEAR = {2017},
     PAGES = {953--993},
      ISSN = {0001-8708},
   MRCLASS = {32Q15 (32U40 32W20 53C44 53C55)},
  MRNUMBER = {3570152},
MRREVIEWER = {Valentino Tosatti},
       DOI = {10.1016/j.aim.2016.10.011},
       URL = {https://doi.org/10.1016/j.aim.2016.10.011},
}

@article {T,
    AUTHOR = {Tian, Gang},
     TITLE = {K-stability and {K}\"ahler-{E}instein metrics},
   JOURNAL = {Comm. Pure Appl. Math.},
  FJOURNAL = {Communications on Pure and Applied Mathematics},
    VOLUME = {68},
      YEAR = {2015},
    NUMBER = {7},
     PAGES = {1085--1156},
      ISSN = {0010-3640},
   MRCLASS = {53C55 (53C25)},
  MRNUMBER = {3352459},
MRREVIEWER = {Matthew B. Stenzel},
       DOI = {10.1002/cpa.21578},
       URL = {http://dx.doi.org/10.1002/cpa.21578},
}

@Article{ DS14,
	title = {{Gromov-Hausdorff limits of K{\"a}hler manifolds and algebraic geometry}},
	author = "Donaldson, Simon and Sun, Song",
	journal = "Acta Math.",
	volume = "213",
	year = "2014", 
	number = "1",
	pages = "63-106"
}

@article {DS17,
    AUTHOR = {Donaldson, Simon and Sun, Song},
     TITLE = {Gromov-{H}ausdorff limits of {K}\"{a}hler manifolds and algebraic
              geometry, {II}},
   JOURNAL = {J. Differential Geom.},
  FJOURNAL = {Journal of Differential Geometry},
    VOLUME = {107},
      YEAR = {2017},
    NUMBER = {2},
     PAGES = {327--371},
      ISSN = {0022-040X},
   MRCLASS = {53C55 (32Q20 53C23)},
  MRNUMBER = {3707646},
MRREVIEWER = {Stuart James Hall},
       DOI = {10.4310/jdg/1506650422},
       URL = {https://doi.org/10.4310/jdg/1506650422},
}

@Article{ Lu,
	title = "{On holomorphic mappings of complex manifolds.}",
	author = "Y.-C. Lu",
	journal = "J. Diff. Geom.",
	pages = "299--312",
	volume = "2",
	year = "1968"
}

\end{document}